\documentclass[11pt]{amsart}

\title[Duality in FEEC and Hodge Duality on the Sphere]{Duality in Finite Element Exterior Calculus and Hodge Duality on the Sphere}
\author{Yakov Berchenko-Kogan}
\subjclass[2010]{65N30, 58A10}
\keywords{finite element exterior calculus, Hodge duality, differential forms, finite element method}
\commby{Douglas Arnold}
\thanks{Communicated by Douglas Arnold}

\usepackage{amsthm,amssymb,tikz}
\usetikzlibrary{cd}
\usepackage[unicode]{hyperref}

\newtheorem{theorem}{Theorem}[section]
\newtheorem{proposition}[theorem]{Proposition}
\newtheorem{corollary}[theorem]{Corollary}
\newtheorem{lemma}[theorem]{Lemma}
\theoremstyle{definition}
\newtheorem{definition}[theorem]{Definition}
\newtheorem{notation}[theorem]{Notation}
\newtheorem{example}[theorem]{Example}

\newcommand{\abs}[1]{\left\lvert{#1}\right\rvert}

\renewcommand{\P}{\mathcal P}
\newcommand{\oP}{\mathring{\mathcal P}}
\newcommand\ringring[1]{%
  {
   \mathop{\kern0pt #1}\limits^{
     \vbox to-1.85ex{
       \kern-2ex 
       \hbox to 0pt{\hss\normalfont\kern.1em \r{}\kern-.45em \r{}\hss}%
       \vss 
     }
   }
  }
}
\newcommand{\ooP}{\ringring{\mathcal P}}
\newcommand{\R}{\mathbb R}
\newcommand\T{\mathcal T}
\newcommand{\vol}{\mathrm{vol}}
\newcommand{\pp}[2][]{\frac{\partial{#1}}{\partial{#2}}}

\begin{document}

\begin{abstract}
  Finite element exterior calculus refers to the development of finite element methods for differential forms, generalizing several earlier finite element spaces of scalar fields and vector fields to arbitrary dimension $n$, arbitrary polynomial degree $r$, and arbitrary differential form degree $k$. The study of finite element exterior calculus began with the $\mathcal P_r\Lambda^k$ and $\mathcal P_r^-\Lambda^k$ families of finite element spaces on simplicial triangulations. In their development of these spaces, Arnold, Falk, and Winther rely on a duality relationship between $\mathcal P_r\Lambda^k$ and $\mathring{\mathcal P}_{r+k+1}^-\Lambda^{n-k}$ and between $\mathcal P_r^-\Lambda^k$ and $\mathring{\mathcal P}_{r+k}\Lambda^{n-k}$. In this article, we show that this duality relationship is, in essence, Hodge duality of differential forms on the standard $n$-sphere, disguised by a change of coordinates. We remove the disguise, giving explicit correspondences between the $\mathcal P_r\Lambda^k$, $\mathcal P_r^-\Lambda^k$, $\mathring{\mathcal P}_r\Lambda^k$ and $\mathring{\mathcal P}_r^-\Lambda^k$ spaces and spaces of differential forms on the sphere. As a direct corollary, we obtain new pointwise duality isomorphisms between $\mathcal P_r\Lambda^k$ and $\mathring{\mathcal P}_{r+k+1}^-\Lambda^{n-k}$ and between $\mathcal P_r^-\Lambda^k$ and $\mathring{\mathcal P}_{r+k}\Lambda^{n-k}$. These isomorphisms can be implemented via a simple computation, which we illustrate with examples.
\end{abstract}

\maketitle

\section{Introduction}
The finite element method is a tool for solving partial differential equations numerically that approximates solutions to the PDE by functions that are piecewise polynomial with respect to a mesh. In the 1970s and 1980s, Raviart, Thomas, Brezzi, Douglas, Marini, and N\'ed\'elec extended these methods to vector equations, such as Maxwell's equations of electromagnetism \cite{bdm85,n80,n86,rt77}. In the early 2000s, Arnold, Falk, and Winther developed finite element exterior calculus, placing scalar and vector finite element methods under a unifying umbrella of finite element methods for differential forms \cite{afw06}.

Given a simplicial triangulation $\T$, Arnold, Falk, and Winther \cite{afw06} constructed two families of spaces of piecewise polynomial differential forms on $\T$, namely the $\P_r\Lambda^k(\T)$ and $\P_r^-\Lambda^k(\T)$ spaces, where $r$ is the polynomial degree and $k$ is the differential form degree. This task amounts to finding a way to uniquely specify a differential form $b\rvert_T$ on each simplex $T$ of the triangulation $\T$, while at the same time ensuring interelement continuity, in the sense that if $T_1$ and $T_2$ share a face $f$, then the restrictions of $b\rvert_{T_1}$ and $b\rvert_{T_2}$ to $f$ agree. Arnold, Falk, and Winther accomplish this task by specifying $b$ using the values of $\int_fa\wedge b$ for each face $f$ of the triangulation $\T$, where $a$ comes from an appropriately chosen space of polynomial differential forms on $f$. Specifically, for $b\in\P_r\Lambda^k(\T)$, they take $a\in\P^-_{r+k-\dim f}\Lambda^{\dim f-k}(f)$, and for $b\in\P_r^-\Lambda^k(\T)$, they take $a\in\P_{r+k-\dim f-1}\Lambda^{\dim f-k}(f)$. With these choices, given arbitrary desired values for $\int_fa\wedge b$ for all $f$ and $a$, there is a unique differential form $b$ on $\T$ with those values satisfying the interelement continuity condition above. See also \cite[Theorem 3.5]{a13} for a short proof of this result. 

As discussed in their later paper \cite[Theorem 4.3]{afw09}, one can reduce the above task of finding appropriate spaces for $a$ to a much simpler task. Namely, given a single simplex $T^n$ of arbitrary dimension $n$, we consider the spaces $\oP_r\Lambda^k(T^n)$ and $\oP_r^-\Lambda^k(T^n)$ of polynomial differential forms on $T^n$ whose restriction to the boundary $\partial T^n$ vanishes. For $b$ coming from one of those spaces, the task is to find an appropriate space for $a$ such that, given desired values of $\int_{T^n}a\wedge b$ for all $a$, there is a unique $b$ with those values. To accomplish this task, for $b\in\oP_r\Lambda^k(T^n)$, we can take $a\in\P_{r+k-n}^-\Lambda^{n-k}(T^n)$, and for $b\in\oP_r^-\Lambda^k(T^n)$, we can take $a\in\P_{r+k-n-1}\Lambda^{n-k}(T^n)$. As the authors discuss in \cite{afw09}, the choices of spaces for $a$ in \cite{afw06} follow by simply setting $f=T^n$.

As discussed in \cite{afw06}, we should think of the above claim as a duality result. For each $b\in\oP_r\Lambda^k(T^n)$, we have a linear functional $a\mapsto\int_{T^n}a\wedge b$. Hence, each $b\in\oP_r\Lambda^k(T^n)$ determines an element of the dual space $\P_{r+k-n}^-\Lambda^{n-k}(T^n)^*$. Thus, we have a map $\oP_r\Lambda^k(T^n)\to\P_{r+k-n}^-\Lambda^{n-k}(T^n)^*$, and likewise we have a map $\oP_r^-\Lambda^k(T^n)\to\P_{r+k-n-1}\Lambda^{n-k}(T^n)^*$. With this perspective, the statement that we can find a unique $b$ that will give desired values for $\int_{T^n}a\wedge b$ for all $a$ is simply the statement that these maps are isomorphisms.

In this paper, it will be convenient to consider the equivalent isomorphisms $\P_r\Lambda^k(T^n)\to\oP_{r+k+1}^-\Lambda^{n-k}(T^n)^*$ and $\P_r^-\Lambda^k(T^n)\to\oP_{r+k}\Lambda^{n-k}(T^n)^*$, which we can get from the ones above by taking the dual map and reindexing $r$ and $k$. Additionally, rather that claiming that these maps are isomorphisms, we can also take the equivalent perspective of claiming that the bilinear paring $(a,b)\mapsto\int_{T^n}a\wedge b$ is nondegenerate. That is, given a nonzero $a\in\P_r\Lambda^k(T^n)$, there exists a form $b\in\oP_{r+k+1}^-\Lambda^{n-k}(T^n)$ such that $\int_{T^n}a\wedge b>0$, and, conversely, for every $b$ there exists such an $a$, and likewise for the other duality relationship.

\subsection{Main results}
In this article, we show that, with a change of coordinates, we can reveal the duality relationships between $\P_r\Lambda^k(T^n)$ and $\oP_{r+k+1}^-\Lambda^{n-k}(T^n)$ and between $\P_r^-\Lambda^k(T^n)$ and $\oP_{r+k}\Lambda^{n-k}(T^n)$ to be Hodge duality on the standard $n$-sphere combined with multiplication by the bubble function $u_N:=u_1\dotsm u_{n+1}$ defined in Notation \ref{notation:ui}. Specifically, using the change of coordinates, the Hodge star on the sphere $*_{S^n}$, and the bubble function $u_N$, we define an injective tensorial map from $k$-forms $\Lambda^k(T^n)$ to $(n-k)$-forms $\Lambda^{n-k}(T^n)$. We show that, under this map, the image of $\P_r\Lambda^k(T^n)$ is $\oP_{r+k+1}^-\Lambda^{n-k}(T^n)$, and the image of $\P_r^-\Lambda^k(T^n)$ is $\oP_{r+k}\Lambda^{n-k}(T^n)$, giving us natural isomorphisms between these pairs of spaces. This result appears as Corollary \ref{cor:isos}. We remark here that defining a natural map from $\P_r\Lambda^k(T^n)$ to $\oP_{r+k+1}^-\Lambda^{n-k}(T^n)$ is very different from defining a natural map from $\P_r\Lambda^k(T^n)$ to the dual space $\oP_{r+k+1}^-\Lambda^{n-k}(T^n)^*$. As discussed above, the latter map follows immediately from the paring $\int_{T^n}a\wedge b$.

Our map also allows us to strengthen the above statement of nondegeneracy of the pairing $(a,b)\mapsto\int_{T^n}a\wedge b$. Given a nonzero $a\in\P_r\Lambda^k(T^n)$, rather than simply saying that there exists a form $b\in\oP_{r+k+1}^-\Lambda^{n-k}(T^n)$ such that $\int_{T^n}a\wedge b>0$, our isomorphism map explicitly constructs this form $b$. Moreover, the tensoriality of this map means that $b$ only depends on $a$ pointwise, in the sense that the value of $b$ at $x\in T^n$ only depends on the value of $a$ at $x$, not on its values at other points of $T^n$. We have an analogous result for the other pairing, and both of these results appear as Corollary \ref{cor:prduality}.

Computing our map is simple and straightforward to implement, as we illustrate in Examples \ref{eg:prduality} and \ref{eg:prmduality}. We also give Examples \ref{eg:scalarfieldsn} and \ref{eg:scalarfields0}, where we specialize to the classical case of scalar fields, computing our isomorphism in the cases $k=n$ and $k=0$.

The key ingredient is the transformation $x_i=u_i^2$, which sends the unit $n$-sphere $u_1^2+\dotsb+u_{n+1}^2=1$ to the standard $n$-simplex $x_1+\dotsb+x_{n+1}=1$. This transformation induces a correspondence between differential forms on the simplex $T^n$ and differential forms on the sphere $S^n$. In Theorems \ref{thm:TSiso0} and \ref{thm:TSiso}, we determine the spaces of differential forms on the sphere that correspond to the $\P_r\Lambda^k(T^n)$, $\P_r^-\Lambda^k(T^n)$, $\oP_r\Lambda^k(T^n)$, and $\oP_r^-\Lambda^k(T^n)$ spaces of differential forms on the simplex. The previously discussed Corollaries \ref{cor:isos} and \ref{cor:prduality} quickly follow from this characterization. We note that these results are equally valid for an arbitrary simplex by mapping it to the standard simplex using barycentric coordinates.

The isomorphism maps $\P_r\Lambda^k(T^n)\xrightarrow\simeq\oP_{r+k+1}^-\Lambda^{n-k}(T^n)$ and $\P_r^-\Lambda^k(T^n)\xrightarrow\simeq\oP_{r+k}\Lambda^{n-k}(T^n)$ given in this paper are the same maps as the ones given in my earlier preprint \cite{bk18feec}, and to the best of my knowledge are the only such maps in the literature that are defined pointwise. This article can be viewed as providing a new interpretation of these maps in terms of Hodge duality on the $n$-sphere. For a different construction of isomorphism maps between these spaces, see Martin Licht's recent work \cite{l18}. The isomorphisms constructed in \cite[Equations (34) and (35)]{l18} involve decomposing differential forms in $\P_r\Lambda^k(T^n)$ and $\P_r^-\Lambda^k(T^n)$ as linear combinations in terms of canonical spanning sets defined in that paper. In contrast, the isomorphism maps in this article can be written directly in terms of the Hodge star, without needing to decompose the differential forms as linear combinations.

\subsection{Outline of this paper}
We discuss our notation and definitions in Section \ref{sec:preliminaries}. In particular, we use a new definition of $\P_r^-\Lambda^k(T^n)$. We show in Appendix \ref{sec:prmdef} that our definition is equivalent to the definition given by Arnold, Falk, and Winther. We also define several spaces of differential forms on the sphere.

In Section \ref{sec:results}, we present the main results discussed above, along with examples that illustrate them. Specifically, in Theorem \ref{thm:TSiso0}, we give the spaces of differential forms on the sphere that correspond to the $\P_r\Lambda^k(T^n)$, $\P_r^-\Lambda^k(T^n)$, $\oP_r\Lambda^k(T^n)$, and $\oP_r^-\Lambda^k(T^n)$ spaces of differential forms on the simplex via the coordinate transformation $x_i=u_i^2$. Then, in Theorem \ref{thm:TSiso}, we reexpress these spaces of differential forms on the sphere in terms of the Hodge star on the sphere $*_{S^n}$ and the bubble function $u_N$. From Theorem \ref{thm:TSiso}, we quickly prove the aforementioned Corollaries \ref{cor:isos} and \ref{cor:prduality}.

Section \ref{sec:proof} is devoted to the proof of Theorem \ref{thm:TSiso0}. Then, in Section \ref{sec:psm}, we study the $\P_s^-\Lambda^k(S^n)$ spaces of differential forms on the sphere given in Definition \ref{def:pm}, and characterize these spaces in terms of the Hodge star $*_{S^n}$. Likewise, in Section \ref{sec:oops}, we study the $\ooP_s\Lambda^k(S^n)$ spaces given in Definition \ref{def:oop}, and characterize these spaces in terms of the bubble function $u_N$. Together, these results allow us to restate Theorem \ref{thm:TSiso0} solely in terms of the space of polynomial differential forms on the sphere $\P_s\Lambda^k(S^n)$, the Hodge star $*_{S^n}$, and the bubble function $u_N$, without reference to the $\P_s^-\Lambda^k(S^n)$ and $\ooP_s\Lambda^k(S^n)$ spaces, giving us Theorem \ref{thm:TSiso}.

\section{Preliminaries}\label{sec:preliminaries}
In this section, we discuss the concepts that we use in our main results. We begin by setting our notation for differential forms on the simplex $T^n$, the sphere $S^n$, and Euclidean space $\R^{n+1}$. Next, we define the spaces of polynomial differential forms that are the subject of this paper. We discuss how to compute the Hodge star on the sphere $*_{S^n}$, and we extend this operator to differential forms on $\R^{n+1}$. Finally, we define notation for the transformation $x_i=u_i^2$, and we define even and odd differential forms on the sphere.
\subsection{Notation}
\begin{notation}
  Let $T^n$ denote the standard $n$-simplex.
  \begin{equation*}
    T^n=\{x\in\R^{n+1}\mid x_i\ge0\text{ for all }i,\text{ and }x_1+\dotsb+x_{n+1}=1\}.
  \end{equation*}
  Let $S^n$ denote the unit sphere.
  \begin{equation*}
    S^n=\{u\in\R^{n+1}\mid u_1^2+\dotsb+u_{n+1}^2=1\}.
  \end{equation*}
  Let $S^n_{>0}$ denote the part of the unit sphere with strictly positive coordinates.
  \begin{equation*}
    S^n_{>0}=\{u\in S^n\mid u_i>0\text{ for all }i\}.
  \end{equation*}
  Define $T^n_{>0}$ similarly.
\end{notation}

\begin{notation}
  Let $\Lambda^k(T^n)$ denote the space of differential $k$-forms on $T^n$. If $a\in\Lambda^k(T^n)$ and $x\in T^n$, let $a_x$ denote $a$ evaluated at $x$. That is, $a_x$ is an antisymmetric $k$-linear tensor on the tangent space $T_xT^n$.

  Likewise, for $\alpha\in\Lambda^k(S^n)$ and $u\in S^n$ we have $\alpha_u$, an antisymmetric $k$-linear tensor on $T_uS^n$. Similarly, for $\hat\alpha\in\Lambda^k(\R^{n+1})$ and $u\in\R^{n+1}$, we have $\hat\alpha_u$, an antisymmetric $k$-linear tensor on $T_u\R^{n+1}\cong\R^{n+1}$.
\end{notation}

\begin{notation}\label{notation:ui}
  For $I=\{i_1<i_2<\dotsb<i_k\}\subseteq\{1,\dotsc,n+1\}$, let $u_I$ denote the product $u_{i_1}\dotsm u_{i_k}$ and let $du_I$ denote the wedge product $du_{i_1}\wedge\dotsb\wedge du_{i_k}$. Let $N=\{1,\dotsc,n+1\}$, so
  \begin{equation*}
    u_N:=u_1\dotsm u_{n+1}.
  \end{equation*}
  In the literature, the function $u_N$ is called a \emph{bubble function}.
\end{notation}

\begin{notation}
  Given a vector field $V$ on $\R^{n+1}$ and $\hat\alpha\in\Lambda^k(\R^{n+1})$, we will let $i_V\hat\alpha\in\Lambda^{k-1}(\R^{n+1})$ denote the \emph{interior product} of $\hat\alpha$ with $V$.

  We will use the same notation for the interior product on $T^n$ and $S^n$.
\end{notation}

\begin{notation}
  Let $\hat a\in\Lambda^k(\R^{n+1})$. Pulling back $\hat a$ via the inclusion map $T^n\hookrightarrow\R^{n+1}$, we obtain a differential form $a\in\Lambda^k(T^n)$. Following standard terminology, we refer to $a$ as the \emph{restriction} of $\hat a$ to $T^n$. We will also say that $\hat a$ is an \emph{extension} of $a$.

  We will also use this terminology for other inclusions, such as $S^n\hookrightarrow\R^{n+1}$.
\end{notation}

\subsection{Spaces of polynomial differential forms}
\begin{definition}\label{def:pr}
  Let $\P_r\Lambda^k(\R^{n+1})$ denote the space of differential $k$-forms on $\R^{n+1}$ whose coefficients are polynomials of degree at most $r$. Let $\P_r\Lambda^k(T^n)$ and $\P_r\Lambda^k(S^n)$ denote the restrictions of $\P_r\Lambda^k(\R^{n+1})$ to $T^n$ and $S^n$, respectively. By convention, $\P_r\Lambda^k(\R^{n+1})=\{0\}$ if $r<0$.
\end{definition}

One must be careful with the definition of $\P_r\Lambda^k(S^n)$, as illustrated by the following example.

\begin{example}\label{eg:prsn}
  Using the notation $(u,v,w)$ to represent a point in $\R^3$, we can consider the polynomial $u^3+uv^2+uw^2$. Despite having degree three, this function is in $\P_1\Lambda^0(S^n)$ because, on $S^n$, we have $u^3+uv^2+uw^2=u(u^2+v^2+w^2)=u$, which has degree one. In other words, as a function on $S^n$, $u^3+uv^2+uw^2$ is the restriction of a degree one polynomial on $\R^{n+1}$ to $S^n$.
\end{example}

\begin{definition}
  Let $X$ be the radial vector field in $\R^{n+1}$.
  \begin{equation*}
    X=\sum_{i=1}^{n+1}x_i\pp{x_i}.
  \end{equation*}
  When necessary to avoid confusion, we will also denote the radial vector field by
  \begin{equation*}
    U=\sum_{i=1}^{n+1}u_i\pp{u_i}.
  \end{equation*}
\end{definition}
Note that via the transformation $x_i=u_i^2$, we have $U=2X$.

\begin{definition}\label{def:pm}
  Let $\P_r^-\Lambda^k(\R^{n+1})$ denote those forms that are in the image under $i_X$ of forms of lower degree. That is,
  \begin{equation*}
    \P_r^-\Lambda^k(\R^{n+1}):=i_X\P_{r-1}\Lambda^{k+1}(\R^{n+1}).
  \end{equation*}
  Let $\P_r^-\Lambda^k(T^n)$ and $\P_r^-\Lambda^k(S^n)$ denote the restrictions of $\P_r^-\Lambda^k(\R^{n+1})$ to $T^n$ and $S^n$, respectively.
\end{definition}

Note that this definition of $\P_r^-\Lambda^k(T^n)$ differs from that of Arnold, Falk, and Winther \cite{afw06}. We show the equivalence of the two definitions in Proposition \ref{prop:prmdfn}.

We now discuss two different notions of ``vanishing trace''.

\begin{definition}\label{def:op}
  Let $\oP_r\Lambda^k(T^n)$ and $\oP_r^-\Lambda^k(T^n)$ denote those forms in $\P_r\Lambda^k(T^n)$ and $\P_r^-\Lambda^k(T^n)$, respectively, whose restriction to $\partial T$ vanishes. That is, $a\in\oP_r\Lambda^k(T^n)$ if $a_x(V_1,\dotsc,V_k)$ for any $x\in\partial T^n$ and any $V_1,\dotsc,V_k$ tangent to $\partial T$.
\end{definition}

\begin{definition}\label{def:oop}
  Let $\Gamma_i\subset\R^{n+1}$ denote the hyperplane defined by $u_i=0$, and let $\Gamma=\bigcup_{i=1}^{n+1}\Gamma_i$. Let $\ooP_s\Lambda^k(\R^{n+1})$ denote those forms $\hat\alpha\in\P_s\Lambda^k(\R^{n+1})$ such that $\hat\alpha_u=0$ for all $u\in\Gamma$. Likewise, let $\ooP_s\Lambda^k(S^n)$ denote those forms $\alpha$ in $\P_s\Lambda^k(S^n)$ such that $\alpha_u=0$ for all $u\in S^n\cap\Gamma$. We define $\ooP_s^-\Lambda^k(\R^{n+1})$ and $\ooP_s^-\Lambda^k(S^n)$ similarly.
\end{definition}

Note that saying that $\hat\alpha$ is in $\ooP_s\Lambda^k(\R^{n+1})$ is stronger than simply saying that the restriction of $\hat\alpha$ to each $\Gamma_i$ vanishes. Saying that $\hat\alpha\in\ooP_s\Lambda^k(\R^{n+1})$ means that $\hat\alpha_u(V_1,\dotsc,V_k)=0$ for any $u\in\Gamma_i$ and \emph{arbitrary} vectors $V_1,\dotsc,V_k\in T_u\mathbb R^{n+1}$, not just vectors tangent to $\Gamma_i$ as with restriction.

\begin{example}
  Using coordinates $(u,v)$ on $\R^2$, consider $\hat\alpha=u\,dv\in\P_1\Lambda^1(\R^2)$.  Observe that the restriction of $\hat\alpha$ to $\{v=0\}$ vanishes because the tangent space of $\{v=0\}$ is spanned by $\pp u$, and $u\,dv(\pp u)=0$. However, $\hat\alpha\notin\ooP_1\Lambda^1(\R^2)$ because it does not vanish on vectors normal to $\{v=0\}$: we have $u\,dv(\pp v)=u$, which is not identically zero on the line $\{v=0\}$.
\end{example}

Note also that if $\hat\alpha\in\P_s\Lambda^k(\R^{n+1})$ is an extension of $\alpha\in\P_s\Lambda^k(S^n)$, checking that $\alpha_u=0$ only involves checking that $\alpha_u(V_1,\dotsc,V_k)=0$ for vectors tangent to $S^n$, whereas checking that $\hat\alpha_u=0$ involves checking that $\hat\alpha_u(V_1,\dotsc,V_k)=0$ for all vectors.

\begin{example}\label{eg:udu}
  Using coordinates $(u,v)$ on $\R^2$, consider $\hat\alpha=v\,dv\in\P_1\Lambda^1(\R^2)$, and let $\alpha$ be the restriction of $\hat\alpha$ to $S^1=\{(u,v)\mid u^2+v^2=1\}$.
  We claim that, on the circle, we have that $\alpha\in\ooP_1\Lambda^1(S^1)$. Indeed, $u^2+v^2=1$ implies that $u\,du+v\,dv=0$, so $\alpha=-u\,du$. Thus $\alpha$ vanishes when we set $u=0$ or $v=0$.
  In contrast, $\hat\alpha\notin\ooP_1\Lambda^1(\R^2)$, since $\hat\alpha$ does not vanish when $u=0$. 
\end{example}

\subsection{The Hodge star}
\begin{notation}
  Let $*_{S^n}\colon\Lambda^k(S^n)\to\Lambda^{n-k}(S^n)$ denote the Hodge star with respect to the standard metric on $S^n$, and let $*_{\R^{n+1}}\colon\Lambda^k(\R^{n+1})\to\Lambda^{n+1-k}(\R^{n+1})$ denote the standard Hodge star on $\R^{n+1}$.
\end{notation}

As we will see, computing $*_{S^n}\alpha$ is a straightforward computation.
\begin{definition}
  Let $\nu$ denote the outward unit conormal to the sphere, so
  \begin{equation*}
    \nu=\sum_{i=1}^{n+1}u_i\,du_i.
  \end{equation*}
\end{definition}

\begin{proposition}\label{prop:spherehodgestar}
  Let $\alpha\in\Lambda^k(S^n)$. Let $\hat\alpha\in\Lambda^k(\R^{n+1})$ be an extension of $\alpha$. Then $*_{S^n}\alpha$ is the restriction to $S^n$ of
  \begin{equation*}
    *_{\R^{n+1}}(\nu\wedge\hat\alpha).
  \end{equation*}
  \begin{proof}
    Proposition \ref{prop:hodgehyperplane} gives the corresponding result for hyperplanes of oriented inner product spaces. We apply it to the hyperplane $T_uS^n\subset T_u\R^{n+1}$.
  \end{proof}
\end{proposition}

\begin{example}
  If $n=2$ and $k=1$, we can think of $\alpha\in\Lambda^1(S^2)$ as a vector field on the sphere. With this interpretation, $*_{S^2}\alpha$ rotates $\alpha$ by $90^\circ$ counterclockwise, and Proposition \ref{prop:spherehodgestar} states that one way to compute $*_{S^2}\alpha$ is to take the cross product of the normal vector with $\alpha$.
\end{example}

Proposition \ref{prop:spherehodgestar} motivates extending the definition of $*_{S^n}\colon\Lambda^k(S^n)\to\Lambda^{n-k}(S^n)$ to $*_{S^n}\colon\Lambda^k(\R^{n+1})\to\Lambda^{n-k}(\R^{n+1})$ as follows.

\begin{definition}
  For $\hat\alpha\in\Lambda^k(\R^{n+1})$, let $*_{S^n}\hat\alpha\in\Lambda^{n-k}(\R^{n+1})$ denote
  \begin{equation*}
    *_{S^n}\hat\alpha:=*_{\R^{n+1}}(\nu\wedge\hat\alpha).
  \end{equation*}
\end{definition}
With this definition, Proposition \ref{prop:spherehodgestar} states that if $\alpha\in\Lambda^k(S^n)$ is the restriction of $\hat\alpha$, then $*_{S^n}\alpha$ is the restriction of $*_{S^n}\hat\alpha$.

\subsection{The coordinate transformation between $S^n$ and $T^n$}
\begin{definition}
  Consider the transformation $\Phi\colon\R^{n+1}\to\R^{n+1}$ defined by
  \begin{equation*}
    \Phi(u)=(u_1^2,\dotsc,u_{n+1}^2)=:(x_1,\dotsc,x_{n+1})=x.
  \end{equation*}
\end{definition}
Observe that $x\in T^n$ if and only if $u\in S^n$. In fact, $\Phi$ is a diffeomorphism when restricted to $S^n_{>0}\to T^n_{>0}$.

We will use the notation $\Phi^*$ to refer to the pullback map both in the context of $\Phi\colon\R^{n+1}\to\R^{n+1}$ and in the context of $\Phi\colon S^n\to T^n$, so we have pullback maps $\Phi^*\colon\Lambda^k(\R^{n+1})\to\Lambda^k(\R^{n+1})$ and $\Phi^*\colon\Lambda^k(T^n)\to\Lambda^k(S^n)$.

Observe that if $a\in\Lambda^k(T^n)$ is the restriction of $\hat a\in\Lambda^k(\R^{n+1})$ to $T^n$, then $\Phi^*a\in\Lambda^k(S^n)$ is the restriction of $\Phi^*\hat a\in\Lambda^k(\R^{n+1})$ to $S^n$.

\subsection{Even and odd differential forms}

Because $\Phi(u_1,\dotsc,u_i,\dotsc,u_{n+1})=\Phi(u_1,\dotsc,-u_i,\dotsc,u_{n+1})$, any differential form in the image of $\Phi^*$ is invariant under all coordinate reflections. We call such forms \emph{even in all variables}, or simply \emph{even}.
\begin{definition}\label{def:evenodd}
  Let the $i$th coordinate reflection $R_i\colon\R^{n+1}\to\R^{n+1}$ be the map
  \begin{equation*}
    R_i(u_1,\dotsc,u_i,\dotsc,u_{n+1})=(u_1,\dotsc,-u_i,\dotsc,u_{n+1}).
  \end{equation*}
  We call $\hat\alpha\in\Lambda^k(\R^{n+1})$ \emph{even in all variables} or simply \emph{even} if $R_i^*\hat\alpha=\hat\alpha$ for every $i$. Similarly, we call $\hat\alpha\in\Lambda^k(\R^{n+1})$ \emph{odd in all variables} or simply \emph{odd} if $R_i^*\hat\alpha=-\hat\alpha$ for every $i$. We denote these spaces by $\Lambda^k_e(\R^{n+1})$ and $\Lambda^k_o(\R^{n+1})$, respectively.

  We define $\Lambda^k_e(S^n)$ and $\Lambda^k_o(S^n)$ similarly. We will also use the notation $\P_s\Lambda^k_e(S^n)$ and $\P_s\Lambda^k_o(S^n)$ to denote the even and odd forms in $\P_s\Lambda^k(S^n)$, respectively, and likewise for the other spaces of polynomial differential forms.
\end{definition}

As with functions in one variable, we can take even and odd parts.
\begin{definition}
  If $\hat\alpha\in\Lambda^k(\R^{n+1})$, let the \emph{even part} $\hat\alpha_e\in\Lambda^k_e(\R^{n+1})$ of $\hat\alpha$ denote the average of all possible reflections of $\hat\alpha$. That is,
  \begin{equation}\label{eq:even}
    \hat\alpha_e=\frac1{2^{n+1}}\sum_{\epsilon_1=0}^1\dotsb\sum_{\epsilon_{n+1}=0}^1(R_1^*)^{\epsilon_1}\dotsm(R_{n+1}^*)^{\epsilon_{n+1}}\hat\alpha.
  \end{equation}
  Likewise, let the \emph{odd part} of $\hat\alpha_o\in\Lambda^k_o(\R^{n+1})$ of $\hat\alpha$ denote the signed average
  \begin{equation}\label{eq:odd}
    \hat\alpha_o=\frac1{2^{n+1}}\sum_{\epsilon_1=0}^1\dotsb\sum_{\epsilon_{n+1}=0}^1(-1)^{\epsilon_1+\dotsb+\epsilon_{n+1}}(R_1^*)^{\epsilon_1}\dotsm(R_{n+1}^*)^{\epsilon_{n+1}}\hat\alpha.
  \end{equation}
  We use the same equations to define the even and odd parts of a differential form on the sphere, $\alpha\in\Lambda^k(S^n)$.
\end{definition}
If $\hat\alpha$ is a differential form with polynomial coefficients, $\hat\alpha_e$ simply extracts those terms that are even $k$-forms, and $\hat\alpha_o$ extracts those terms that are odd $k$-forms. Note that $\hat\alpha$ is not the sum of its even and odd parts; there will generally be terms that are even in some variables and odd in others.

\begin{proposition}\label{prop:restrictevenodd}
  Restriction commutes with taking even and odd parts. That is, if $\alpha\in\Lambda^k(S^n)$ is the restriction of $\hat\alpha\in\Lambda^k(\R^{n+1})$ to $S^n$, then $\alpha_e$ and $\alpha_o$ are the restrictions of $\hat\alpha_e$ and $\hat\alpha_o$, respectively.
  \begin{proof}
    If $\alpha$ is the restriction of $\hat\alpha$, then $R_i^*\alpha$ is the restriction of $R_i^*\hat\alpha$. Thus, when we restrict equations \eqref{eq:even} and \eqref{eq:odd} to $S^n$, we obtain the corresponding equations for $\alpha_e$ and $\alpha_o$.
  \end{proof}
\end{proposition}

In the following proposition, we summarize how the operations we have considered interact with even and odd parts.

\begin{proposition}\label{prop:starevenodd}\label{prop:nuevenodd}\label{prop:starsphereevenodd}\label{prop:xevenodd}\label{prop:unevenodd}
  The following operations preserve even and odd parts.
  \begin{align*}
    \nu\wedge\hat\alpha_e&=(\nu\wedge\hat\alpha)_e,&\nu\wedge\hat\alpha_o&=(\nu\wedge\hat\alpha)_o,\\
    i_U\hat\alpha_e&=(i_U\hat\alpha)_e,&i_U\hat\alpha_o&=(i_U\hat\alpha)_o.
  \end{align*}

  The following operations interchange even and odd parts.
  \begin{align*}
    *_{\R^{n+1}}\hat\alpha_e&=(*_{\R^{n+1}}\hat\alpha)_o,&*_{\R^{n+1}}\hat\alpha_o&=(*_{\R^{n+1}}\hat\alpha)_e,\\
    *_{S^n}\hat\alpha_e&=(*_{S^n}\hat\alpha)_o,&*_{S^n}\hat\alpha_o&=(*_{S^n}\hat\alpha)_e,\\
    u_N\hat\alpha_e&=(u_N\hat\alpha)_o,&u_N\hat\alpha_o&=(u_N\hat\alpha)_e.
  \end{align*}
\end{proposition}
\begin{proof}
  Observe that
  \begin{align*}
    R_i^*(\nu\wedge\hat\alpha)&=R_i^*\nu\wedge R_i^*\hat\alpha=\nu\wedge R_i^*\hat\alpha,\\
    R_i^*(i_U\hat\alpha)&=i_{R_i^*U}(R_i^*\hat\alpha)=i_U(R_i^*\hat\alpha),\\
    R_i^*(u_N\hat\alpha)&=(R_i^*u_N)(R_i^*\hat\alpha)=-u_N(R_i^*\hat\alpha).
  \end{align*}
  Reflections reverse orientation, so
  \begin{equation*}
    *_{\R^{n+1}}(R_i^*\hat\alpha)=-R_i^*(*_{\R^{n+1}}\hat\alpha).
  \end{equation*}
  Consequently, since $*_{S^n}\hat\alpha=*_{\R^{n+1}}(\nu\wedge\hat\alpha)$,
  \begin{equation*}
    *_{S^n}(R_i^*\hat\alpha)=-R_i^*(*_{S^n}\hat\alpha).
  \end{equation*}
  The result follows by applying these operations to equations \eqref{eq:even} and \eqref{eq:odd}.
\end{proof}

\section{Results and examples}\label{sec:results}
We state our main result that the coordinate transformation $\Phi$ induces correspondences between spaces of polynomial differential forms on $T^n$ and polynomial differential forms on $S^n$. As a consequence of these correspondences, we easily obtain the duality relationships between the $\P$ and $\P^-$ spaces. We work with the standard simplex, but all of our results apply equally well to an arbitrary simplex, by using barycentric coordinates.

\begin{theorem}\label{thm:TSiso0}
The map $\Phi^*\colon\Lambda^k(T^n)\to\Lambda^k(S^n)$ induced by the coordinate transformation $\Phi$ gives the following correspondences between polynomial differential forms on the simplex and polynomial differential forms on the sphere.
\begin{align*}
  \P_r\Lambda^k(T^n)&\xrightarrow\simeq\P_{2r+k}\Lambda^k_e(S^n),\\
  \P_r^-\Lambda^k(T^n)&\xrightarrow\simeq\P_{2r+k}^-\Lambda^k_e(S^n),\\
  \oP_r\Lambda^k(T^n)&\xrightarrow\simeq\ooP_{2r+k}\Lambda^k_e(S^n),\\
  \oP_r^-\Lambda^k(T^n)&\xrightarrow\simeq\ooP_{2r+k}^-\Lambda^k_e(S^n),
\end{align*}
where these spaces of differential forms on $S^n$ are defined in Definitions \ref{def:pr}, \ref{def:pm}, \ref{def:oop}, and \ref{def:evenodd}.
\end{theorem}
\begin{proof}
  We prove this theorem in Section \ref{sec:proof}. Specifically, the four isomorphisms are proved in Theorems \ref{thm:psphere}, \ref{thm:pmsphere}, \ref{thm:vanishingtraceTS}, and \ref{thm:oopmTS}, respectively.
\end{proof}

The results of Sections \ref{sec:psm} and \ref{sec:oops} allow us to express the above spaces of differential forms on $S^n$ solely in terms of $\P_s\Lambda^k_{e/o}(S^n)$, the Hodge star $*_{S^n}$, and the bubble function $u_N$. We obtain the following version of the theorem.

\begin{theorem}\label{thm:TSiso}
The map $\Phi^*\colon\Lambda^k(T^n)\to\Lambda^k(S^n)$ induced by the coordinate transformation $\Phi$ gives the following correspondences between polynomial differential forms on the simplex and polynomial differential forms on the sphere.
\begin{align*}
  \P_r\Lambda^k(T^n)&\xrightarrow\simeq\P_{2r+k}\Lambda^k_e(S^n),\\
  \P_r^-\Lambda^k(T^n)&\xrightarrow\simeq *_{S^n}\P_{2r+k-1}\Lambda^{n-k}_o(S^n),\\
  \oP_r\Lambda^k(T^n)&\xrightarrow\simeq u_N\P_{2r+k-n-1}\Lambda^k_o(S^n),\quad r\ge1,\\
  \oP_r^-\Lambda^k(T^n)&\xrightarrow\simeq u_N{*_{S^n}\P_{2r+k-n-2}}\Lambda^{n-k}_e(S^n).
\end{align*}
\end{theorem}
\begin{proof}
  The first isomorphism is the same as in Theorem \ref{thm:TSiso0}. In light of Theorem \ref{thm:TSiso0}, in order to obtain the second isomorphism, we must show that $\P_s^-\Lambda^k_e(S^n)=*_{S^n}\P_{s-1}\Lambda^{n-k}_o(S^n)$. We prove this claim in Corollary \ref{cor:pmhodgesphereevenodd}. To prove the third isomorphism, we must show that $\ooP_s\Lambda^k_e(S^n)=u_N\ooP_{s-n-1}\Lambda^k_o(S^n)$ if $s\ge2+k$. This claim follows from Corollary \ref{cor:oopeven}. Finally, we can prove the last isomorphism by showing that $\ooP_s^-\Lambda^k_e(S^n)=u_N\P_{s-n-1}^-\Lambda^k_o(S^n)$ and $\P_{s-n-1}^-\Lambda^k_o(S^n)=*_{S^n}\P_{s-n-2}\Lambda^{n-k}_e(S^n)$. We prove the first claim in Corollary \ref{cor:oopmeven}, and the second claim follows from Corollary \ref{cor:pmhodgesphereevenodd}.
\end{proof}

As an immediate corollary, we obtain an explicit pointwise construction of the duality isomorphisms of Arnold, Falk, and Winther.
\begin{corollary}\label{cor:isos}
  The pointwise-defined map $(\Phi^*)^{-1}\circ(u_N*_{S^n})\circ\Phi^*$ is an isomorphism between the following spaces.
  \begin{equation*}
    \begin{tikzcd}[column sep = 3cm]
      \P_r\Lambda^k(T^n)\arrow[r, shift left, "(\Phi^*)^{-1}\circ(u_N*_{S^n})\circ\Phi^*"]&\arrow[l, shift left]\oP_{r+k+1}^-\Lambda^{n-k}(T^n),\\
      \P_r^-\Lambda^k(T^n)\arrow[r, shift left, "(\Phi^*)^{-1}\circ(u_N*_{S^n})\circ\Phi^*"]&\arrow[l, shift left]\oP_{r+k}\Lambda^{n-k}(T^n).
    \end{tikzcd}
  \end{equation*}
  These spaces are isomorphic via pointwise-defined maps to $\P_{2r+k}\Lambda^k_e(S^n)$ and $\P_{2r+k-1}\Lambda^{n-k}_o(S^n)$, respectively.
  \begin{proof}
    These isomorphisms follow directly from Theorem \ref{thm:TSiso} along with the fact that $*_{S^n}(*_{S^n}\alpha)=(-1)^{k(n-k)}\alpha$, so $*_{S^n}(*_{S^n}\P_{2r+k-1}\Lambda^{n-k}_o(S^n))=\P_{2r+k-1}\Lambda^{n-k}_o(S^n)$. The operations of pullback, Hodge star, and multiplication by a scalar function are all pointwise-defined maps.
  \end{proof}
\end{corollary}

We show that these isomorphisms give forms that are dual to one another with respect to integration.

\begin{corollary}\label{cor:prduality}
  Consider nonzero $a\in\P_r\Lambda^k(T^n)$ and $b\in\oP_{r+k+1}^-\Lambda^{n-k}(T^n)$ corresponding to one another via the first isomorphism in Corollary \ref{cor:isos}. Then $\int_{T^n}a\wedge b>0$. Moreover, $a$ and $b$ depend on each other only pointwise, in the sense that for $x\in T^n$, $b_x$ depends only on $a_x$ and vice versa.

  The same holds for $a\in\P_r^-\Lambda^k(T^n)$ and $b\in\oP_{r+k}\Lambda^{n-k}(T^n)$ corresponding to one another via the second isomorphism in Corollary \ref{cor:isos}.
  \begin{proof}
    Let $\alpha=\Phi^*a$ and $\beta=\Phi^*b$, so $\beta=u_N(*_{S^n}\alpha)$. By $u$-substitution, we have
    \begin{multline*}
      \int_{T^n}a\wedge b=\int_{T^n_{>0}}a\wedge b=\int_{S^n_{>0}}\alpha\wedge\beta\\=\int_{S^n_{>0}}\alpha\wedge u_N\left(*_{S^n}\alpha\right)=\int_{S^n_{>0}}u_N\langle\alpha,\alpha\rangle_{S^n}>0
    \end{multline*}
    because $\alpha$ is not identically zero and $u_N>0$ on $S^n_{>0}$.
  \end{proof}
\end{corollary}

We provide an example of each of the two isomorphisms in the case $n=2$. To simplify notation, we use coordinates $(x,y,z)$ and $(u,v,w)$ on $\R^3$. We compute the Hodge star $*_{S^n}$ using the formula in Proposition \ref{prop:spherehodgestar}. In the case of one-forms on the two-sphere as in these examples, the Hodge star can be interpreted as rotating a vector field on the sphere $90$ degrees counterclockwise, and the formula in Proposition \ref{prop:spherehodgestar} states that one can do so by taking the cross product of the unit normal with the vector field.

\begin{example}\label{eg:prduality}
  \begin{align*}
    a&=y\,dy&&\in\P_1\Lambda^1(T^2),\\
    \alpha&=2v^3\,dv&&\in\P_3\Lambda^1_e(S^2),\\
    *_{S^2}\alpha
    &=2uv^3\,dw-2v^3w\,du&&\in *_{S^2}\P_3\Lambda^1_e(S^2)\\
    \beta&=2u^2v^4w\,dw-2uv^4w^2\,du&&\in(uvw){*_{S^2}\P_3\Lambda^1_e(S^2)},\\
    b&=xy^2\,dz-y^2z\,dx&&\in\oP_3^-\Lambda^1(T^2).
  \end{align*}
\end{example}

\begin{example}\label{eg:prmduality}
  \begin{align*}
    a&=x\,dy-y\,dx&&\in\P_1^-\Lambda^1(T^2),\\
    \alpha&=2u^2v\,dv-2uv^2\,du&&\in*_{S^2}\P_2\Lambda^1_o(S^2),\\
    \begin{split}
      *_{S^2}\alpha
      &=2(u^3v+uv^3)dw-2u^2vw\,du-2uv^2w\,dv\\
      &=2uv(u^2+v^2+w^2)\,dw-uvw\,d(u^2+v^2+w^2)\\
      &=2uv\,dw
    \end{split}&&\in\P_2\Lambda^1_o(S^2),\\
    \beta
    &=2u^2v^2w\,dw&&\in(uvw)\P_2\Lambda^1_o(S^2),\\
    b&=xy\,dz&&\in\oP_2\Lambda^1(T^2).
    \end{align*}
\end{example}

In the two examples below, we apply the isomorphism in Corollary \ref{cor:isos} to the classical case of scalar fields, that is, $n$-forms and $0$-forms.

\begin{example}\label{eg:scalarfieldsn}
  Consider a general polynomial $n$-form on $T^n$, which can be expressed as
  \begin{equation*}
    a=p(x_1,\dotsc,x_{n+1})\,dx_2\wedge\dotsb\wedge dx_{n+1}\in\P_r\Lambda^n(T^n)=\P_{r+1}^-\Lambda^n(T^n).
  \end{equation*}
  Then
  \begin{equation*}
    \alpha=2^np(u_1^2,\dotsc,u_{n+1}^2)u_2\dotsm u_{n+1}\,du_2\wedge\dotsb\wedge du_{n+1}.
  \end{equation*}
  With Proposition \ref{prop:spherehodgestar}, we can compute that
  \begin{multline}\label{eq:hodgestarvol}
    *_{S^n}(du_2\wedge\dotsb\wedge du_{n+1})=*_{\R^{n+1}}(\nu\wedge(du_2\wedge\dotsb\wedge du_{n+1}))\\=*_{\R^{n+1}}(u_1\vol_{\R^{n+1}})=u_1.
  \end{multline}
  Thus, we have
  \begin{equation*}
    *_{S^n}\alpha=2^np(u_1^2,\dotsc,u_{n+1}^2)u_2\dotsm u_{n+1}\cdot u_1=2^nu_Np(u_1^2,\dotsc,u_{n+1}^2).
  \end{equation*}
  We conclude that
  \begin{equation*}
    \beta=u_N{*_{S^n}\alpha}=2^nu_N^2p(u_1^2,\dotsc,u_{n+1}^2),
  \end{equation*}
  and so
  \begin{equation*}
    b=2^nx_1\dotsm x_{n+1}p(x_1,\dotsc,x_{n+1})\in\oP_{r+n+1}^-\Lambda^0(T^n)=\oP_{r+n+1}\Lambda^0(T^n).
  \end{equation*}
\end{example}

\begin{example}\label{eg:scalarfields0}
  Consider now a general polynomial $0$-form on $T^n$, expressed as
  \begin{equation*}
    a=p(x_1,\dotsc,x_{n+1})\in\P_r\Lambda^0(T^n)=\P_r^-\Lambda^0(T^n).
  \end{equation*}
  Then we have
  \begin{align*}
    \alpha&=p(u_1^2,\dotsc,u_{n+1}^2),\\
    *_{S^n}\alpha&=p(u_1^2,\dotsc,u_{n+1}^2)\vol_{S^n},\\
    \beta&=p(u_1^2,\dotsc,u_{n+1}^2)u_N\vol_{S^n}.
  \end{align*}
  We can verify that, on the sphere, $u_N\vol_{S^n}=u_2\dotsm u_{n+1}\,du_2\wedge\dotsb\wedge du_{n+1}$. Indeed, using equation \eqref{eq:hodgestarvol}, we see that taking $*_{S^n}$ of both sides gives $u_N$, so these $n$-forms must be equal on the sphere. We conclude that
  \begin{equation*}
    \beta=p(u_1^2,\dotsc,u_{n+1}^2)u_2\dotsm u_{n+1}\,du_2\wedge\dotsb\wedge du_{n+1},
  \end{equation*}
  and so
  \begin{equation*}
    b=\tfrac1{2^n}p(x_1,\dotsc,x_{n+1})\,dx_2\wedge\dotsb\wedge dx_{n+1}\in\oP_{r+1}^-\Lambda^n(T^n)=\oP_r\Lambda^n(T^n).
  \end{equation*}
\end{example}

\section{Proof of Theorem \ref{thm:TSiso0}}\label{sec:proof}
Theorem \ref{thm:TSiso0} claims four isomorphisms. We prove the isomorphism for $\P_r\Lambda^k(T^n)$ in Section \ref{subsec:pr}, for $\P_r^-\Lambda^k(T^n)$ in Section \ref{subsec:prm}, and for $\oP_r\Lambda^k(T^n)$ and $\oP_r^-\Lambda^k(T^n)$ in Section \ref{subsec:oop}.

\subsection{The correspondence for \texorpdfstring{$\mathcal P_r\Lambda^k(T^n)$}{Pr\textLambda k(Tn)}}\label{subsec:pr}
We first prove the result for forms on $\R^{n+1}$, and then use it to conclude the isomorphism between forms on $T^n$ and forms on $S^n$.

\begin{proposition}\label{prop:pr}
  The map $\Phi^*$ is an isomorphism between the following spaces.
  \begin{equation*}
    \begin{tikzcd}
      \P_r\Lambda^k(\R^{n+1})\arrow[r, shift left, "\Phi^*"]&\arrow[l, shift left]\P_{2r+k}\Lambda^k_e(\R^{n+1}).
    \end{tikzcd}
  \end{equation*}
  \begin{proof}
    Assume $\hat a\in\P_r\Lambda^k(\R^{n+1})$, and let $\hat\alpha=\Phi^*\hat a$. Using Notation \ref{notation:ui}, let
    \begin{equation*}
      \hat a=\sum_Ip_I(x_1,\dotsc,x_{n+1})\,dx_I,
    \end{equation*}
    where $p_I$ is a polynomial in $x$ of degree at most $r$. Since $x_i=u_i^2$, we have $dx_i=2u_i\,du_i$, so
    \begin{equation*}
      \hat\alpha=2^k\sum_Ip_I(u_1^2,\dotsc,u_{n+1}^2)u_I\,du_I.
    \end{equation*}
    The coefficient $p_I(u_1^2,\dotsc,u_{n+1}^2)$ is even of degree at most $2r$, and $u_Idu_I\in\P_k\Lambda^k_e(\R^{n+1})$, so $\hat\alpha\in\P_{2r+k}\Lambda^k_e(\R^{n+1})$.

    Conversely, assume $\hat\alpha\in\P_{2r+k}\Lambda^k_e(\R^{n+1})$. Let
    \begin{equation*}
      \hat\alpha=\sum_Iq_I(u_1,\dotsc,u_{n+1})\,du_I.
    \end{equation*}
    Observe that $R_i^*(du_I)=-du_I$ if $i\in I$ and $R_i^*(du_I)=du_I$ if $i\notin I$. Since $R_i^*\hat\alpha=\hat\alpha$, we conclude that $R_i^*q_I=-q_I$ if $i\in I$ and $R_i^*q_I=q_I$ if $i\notin I$. 
    In other words, $q_I$ is odd in the variable $u_i$ for $i\in I$ and even in $u_i$ for $i\notin I$.

    Thus, $q_I$ is divisible by $u_i$ for $i\in I$, so we can write $q_I=2^kr_Iu_I$ for a polynomial $r_I$, including a factor of $2^k$ for convenience. We see that $r_I$ is even in all variables $u_i$ for $1\le i\le n+1$. Thus, we can write $r_I(u_1,\dotsc,u_{n+1})=p_I(u_1^2,\dotsc,u_{n+1}^2)$. We conclude that
    \begin{equation*}
      \hat\alpha=2^k\sum_Ip_I(u_1^2,\dotsc,u_{n+1}^2)u_I\,du_I,
    \end{equation*}
    which is $\Phi^*\hat a$ for $\hat a=\sum_Ip_I\,dx_I$ as shown above.
  \end{proof}
\end{proposition}

We use this result to show that this isomorphism holds between forms on $T^n$ and $S^n$. We begin with a simple lemma.
\begin{lemma}\label{lemma:injective}
  The map $\Phi^*\colon\Lambda^k(T^n)\to\Lambda^k(S^n)$ is injective.
  \begin{proof}
    Let $a\in\Lambda^k(T^n)$, and assume that $\Phi^*a=0$. Because $\Phi\colon S^n_{>0}\to T^n_{>0}$ is a diffeomorphism, we know that $\Phi^*\colon\Lambda^k(T^n_{>0})\to\Lambda^k(S^n_{>0})$ is a bijection, so $a$ must be zero on $T^n_{>0}$. Because $a$ is continuous, it must therefore be zero on all of $T^n$.
  \end{proof}
\end{lemma}

\begin{theorem}[Theorem \ref{thm:TSiso0}, first isomorphism]\label{thm:psphere}
  The map $\Phi^*$ is an isomorphism between the following spaces.
  \begin{equation*}
    \begin{tikzcd}
      \P_r\Lambda^k(T^n)\arrow[r, shift left, "\Phi^*"]&\arrow[l, shift left]\P_{2r+k}\Lambda^k_e(S^n).
    \end{tikzcd}
  \end{equation*}
  \begin{proof}
    Let $a\in\P_r\Lambda^k(T^n)$. By definition, there exists an extension $\hat a\in\P_r\Lambda^k(\R^{n+1})$. Let $\hat\alpha=\Phi^*\hat a$, which is in $\P_{2r+k}\Lambda^k_e(\R^{n+1})$ by Proposition \ref{prop:pr}. Let $\alpha$ be the restriction of $\hat\alpha$ to $S^n$, so we have $\alpha=\Phi^*a$, and $\alpha\in\P_{2r+k}\Lambda^k_e(S^n)$ by definition. We conclude that $\Phi^*$ does indeed map $\P_r\Lambda^k(T^n)$ to $\P_{2r+k}\Lambda^k_e(S^n)$, and we know that this map is injective by Lemma \ref{lemma:injective}.

    To show surjectivity, let $\alpha\in\P_{2r+k}\Lambda^k_e(S^n)$. By definition, there exists an extension $\hat\alpha\in\P_{2r+k}\Lambda^k_e(\R^{n+1})$. By Proposition \ref{prop:pr}, there exists $\hat a\in\P_r\Lambda^k(\R^{n+1})$ such that $\hat\alpha=\Phi^*\hat a$. Letting $a$ be the restriction of $\hat a$ to $T^n$, we see that $a\in\P_r\Lambda^k(T^n)$ and $\alpha=\Phi^*a$.
    
  \end{proof}
\end{theorem}

\subsection{The correspondence for \texorpdfstring{$\mathcal P_r^-\Lambda^k(T^n)$}{Pr-\textLambda k(Tn)}}\label{subsec:prm} As before, we begin by considering forms on $\R^{n+1}$. We first prove that the change of coordinates induces an isomorphism between $\P_r^-\Lambda^k(\R^{n+1})$ and $\P_{2r+k}^-\Lambda^k_e(\R^{n+1})$, and then conclude that it also induces an isomorphism between $\P_r^-\Lambda^k(T^n)$ and $\P_{2r+k}^-\Lambda^k_e(S^n)$.

\begin{proposition}\label{prop:pm}
  The map $\Phi^*$ is an isomorphism between the following spaces.
  \begin{equation*}
    \begin{tikzcd}
      \P_r^-\Lambda^k(\R^{n+1})\arrow[r, shift left, "\Phi^*"]&\arrow[l, shift left]\P_{2r+k}^-\Lambda^k_e(\R^{n+1}).
    \end{tikzcd}
  \end{equation*}
  \begin{proof}
    The key fact needed here is that $\Phi$ preserves the radial vector field up to a constant scalar factor. That is, one can compute that $\Phi_*U=2X$.
    
    If $\hat a\in\P_r^-\Lambda^k(\R^{n+1})$, then by definition $\hat a=i_X\hat b$ for some $\hat b\in\P_{r-1}\Lambda^{k+1}(\R^{n+1})$. Let $\hat\beta=\Phi^*\hat b$, which is in $\P_{2r-2+k+1}\Lambda^{k+1}_e(\R^{n+1})$ by Proposition \ref{prop:pr}. Then
    \begin{equation}\label{eq:pmcoordchange}
      \Phi^*(i_X\hat b)=\tfrac12\Phi^*\left(i_{\Phi_*U}\hat b\right)=\tfrac12i_U(\Phi^*\hat b)=\tfrac12 i_U\hat\beta.
    \end{equation}
    Thus $\Phi^*\hat a=\tfrac12 i_U\hat\beta$. Since $\hat\beta\in\P_{2r+k-1}\Lambda^{k+1}_e(\R^{n+1})$, we know that $\Phi^*\hat a\in\P_{2r+k}^-\Lambda^k_e(\R^{n+1})$ by Definition \ref{def:pm} and Proposition \ref{prop:xevenodd}.

    Conversely, if $\hat\alpha\in\P_{2r+k}^-\Lambda^k_e(\R^{n+1})$, then by definition $\hat\alpha=\tfrac12i_U\hat\beta'$ for some $\hat\beta'\in\P_{2r+k-1}\Lambda^{k+1}(\R^{n+1})$. Let $\hat\beta$ be the even part of $\hat\beta'$. By Proposition \ref{prop:xevenodd}, $\hat\alpha=\hat\alpha_e=\frac12i_U\hat\beta'_e=\frac12i_U\hat\beta$. By Proposition \ref{prop:pr}, $\hat\beta=\Phi^*\hat b$ for some $\hat b\in\P_{r-1}\Lambda^{k+1}(\R^{n+1})$. Set $\hat a=i_X\hat b$. Then $\hat a\in\P_r^-\Lambda^k(\R^{n+1})$ by definition, and $\Phi^*\hat a=\hat\alpha$ by equation \eqref{eq:pmcoordchange}.
  \end{proof}
\end{proposition}

We conclude the corresponding isomorphism between forms on $T^n$ and forms on $S^n$.

\begin{theorem}[Theorem \ref{thm:TSiso0}, second isomorphism]\label{thm:pmsphere}
  The map $\Phi^*$ is an isomorphism between the following spaces.
  \begin{equation*}
    \begin{tikzcd}
      \P_r^-\Lambda^k(T^n)\arrow[r, shift left, "\Phi^*"]&\arrow[l, shift left]\P_{2r+k}^-\Lambda^k_e(S^n).
    \end{tikzcd}
  \end{equation*}
  \begin{proof}
    Assume $a\in\P_r^-\Lambda^k(T^n)$. By definition, $a$ has an extension $\hat a\in\P_r^-\Lambda^k(\R^{n+1})$. Then $\Phi^*\hat a$ is in $\P_{2r+k}^-\Lambda^k_e(\R^{n+1})$ by Proposition \ref{prop:pm}. Since $\Phi^*a$ is the restriction of $\Phi^*\hat a$ to $S^n$, we have that $\Phi^*a\in\P_{2r+k}^-\Lambda^k_e(S^n)$ by definition.

    Conversely, if $\alpha\in\P_{2r+k}^-\Lambda^k_e(S^n)$, then by definition it has an extension $\hat\alpha\in\P_{2r+k}^-\Lambda^k_e(\R^{n+1})$. By Proposition \ref{prop:pm}, $\hat\alpha=\Phi^*\hat a$ for $\hat a\in\P_r^-\Lambda^k(\R^{n+1})$. Letting $a$ be the restriction of $\hat a$ to $T^n$, we have that $a\in\P_r^-\Lambda^k(T^n)$ by definition, and $\alpha=\Phi^*a$.
  \end{proof}
\end{theorem}

\subsection{The correspondence for \texorpdfstring{$\mathring{\mathcal P}_r\Lambda^k(T^n)$}{oPr\textLambda k(Tn)} and \texorpdfstring{$\mathring{\mathcal P}_r^-\Lambda^k(T^n)$}{oPr-\textLambda k(Tn)}}\label{subsec:oop} We now show that forms on $T^n$ with vanishing trace correspond to forms on $S^n$ with vanishing trace. However, the notions of ``vanishing trace'' are different; see Definitions \ref{def:op} and \ref{def:oop} and the discussion and examples that follow.

\begin{theorem}[Theorem \ref{thm:TSiso0}, third isomorphism]\label{thm:vanishingtraceTS}
  The map $\Phi^*$ is an isomorphism between the following spaces.
  \begin{equation*}
    \begin{tikzcd}
      \oP_r\Lambda^k(T^n)\arrow[r, shift left, "\Phi^*"]&\arrow[l, shift left]\ooP_{2r+k}\Lambda^k_e(S^n).
    \end{tikzcd}
  \end{equation*}
  \begin{proof}
    Let $T^{n-1}_i$ denote $T^n\cap\Gamma_i$, and likewise let $S^{n-1}_i$ denote $S^n\cap\Gamma_i$. Let $a\in\oP_r\Lambda^k(T^n)$, and let $\alpha=\Phi^*a$. We know that $\alpha\in\P_{2r+k}\Lambda^k_e(S^n)$, and we aim to show that $\alpha\in\ooP_{2r+k}\Lambda^k_e(S^n)$. Fix $1\le i\le n+1$, and let $\bar a$ and $\bar\alpha$ denote the restrictions of $a$ and $\alpha$ to $T^{n-1}_i$ and $S^{n-1}_i$, respectively.

    By definition, $\bar a=0$. Since $\Phi$ maps $S^{n-1}_i$ to $T^{n-1}_i$, we have that $\bar\alpha=\Phi^*\bar a=0$. In other words, we know that $\alpha$ vanishes on vectors tangent to $S^{n-1}_i$. It remains to to show that $\alpha$ vanishes on the vector $\pp{u_i}$ normal to $S^{n-1}_i$, which we can do using the fact that $\alpha$ is even. More precisely, let $\overline{i_{\pp{u_i}}\alpha}$ denote the restriction of the interior product $i_{\pp{u_i}}\alpha$ to $S^{n-1}_i$. We must check $\overline{i_{\pp{u_i}}\alpha}=0$.

    Applying the reflection $R_i$, we have that $R_i^*\bigl(\pp{u_i}\bigr)=-\pp{u_i}$, so $R_i^*\Bigl(i_{\pp{u_i}}\alpha\Bigr)=-i_{\pp{u_i}}\alpha$ because $\alpha$ is even. On the other hand, the reflection $R_i$ fixes $S^{n-1}_i$, so the restrictions of $R_i^*\Bigl(i_{\pp{u_i}}\alpha\Bigr)$ and $i_{\pp{u_i}}\alpha$ to $S^{n-1}_i$ are equal. We conclude that $\overline{i_{\pp{u_i}}\alpha}=-\overline{i_{\pp{u_i}}\alpha}$, so $\overline{i_{\pp{u_i}}\alpha}=0$, as desired.

    Conversely, assume that $\alpha\in\ooP_{2r+k}\Lambda^k_e(S^n)$. We know that $\alpha=\Phi^*a$ for $a\in\P_r\Lambda^k(T^n)$, and we aim to show that $a\in\oP_r\Lambda^k(T^n)$. We have in particular that the restriction of $\alpha$ to $S^{n-1}_i$ is zero, that is,  $\bar\alpha=0$ in the above notation. Thus, $\Phi^*\bar a=\bar\alpha=0$. Applying Lemma \ref{lemma:injective} to $\Phi^*\colon\Lambda^k(T^{n-1}_i)\to\Lambda^k(S^{n-1}_i)$, we conclude that $\bar a=0$. Thus, the restriction of $a$ to $T^{n-1}_i$ vanishes for all $i$, so we conclude that $a\in\oP_r\Lambda^k(T^n)$, as desired.
  \end{proof}
\end{theorem}

\begin{theorem}[Theorem \ref{thm:TSiso0}, fourth isomorphism]\label{thm:oopmTS}
  The map $\Phi^*$ is an isomorphism between the following spaces.
  \begin{equation*}
    \begin{tikzcd}
      \oP_r^-\Lambda^k(T^n)\arrow[r, shift left, "\Phi^*"]&\arrow[l, shift left]\ooP_{2r+k}^-\Lambda^k_e(S^n).
    \end{tikzcd}
  \end{equation*}
  \begin{proof}
    Theorem \ref{thm:vanishingtraceTS} tells us that $\Phi^*$ is an isomorphism between $\oP_r\Lambda^k(T^n)$ and $\ooP_{2r+k}\Lambda^k_e(S^n)$, and Theorem \ref{thm:pmsphere} tells us that $\Phi^*$ is an isomorphism between $\P_r^-\Lambda^k(T^n)$ and $\P_{2r+k}^-\Lambda^k_e(S^n)$. Taking the intersection, we obtain the desired result.
  \end{proof}
\end{theorem}

\section{Characterizations of \texorpdfstring{$\P_s^-\Lambda^k(S^n)$}{Ps-\textLambda k(Sn)}}\label{sec:psm} In Section \ref{subsec:prm}, we showed that $\Phi^*$ is an isomorphism between $\P_r^-\Lambda^k(T^n)$ and $\P_{2r+k}^-\Lambda^k_e(S^n)$. To complete the proof of the correspondence in Theorem \ref{thm:TSiso} for the $\P_r^-\Lambda^k(T^n)$ spaces, we show in this section that $\P_{2r+k}^-\Lambda^k_e(S^n)=*_{S^n}\P_{2r+k-1}\Lambda^{n-k}_o(S^n)$. As before, we first begin with the corresponding result on $\R^{n+1}$.

\begin{proposition}\label{prop:pmhodge}
  The space $\P_s^-\Lambda^k(\R^{n+1})$ is the image under $*_{S^n}$ of differential forms of lower polynomial degree. That is,
  \begin{equation*}
    \P_s^-\Lambda^k(\R^{n+1})=*_{S^n}\P_{s-1}\Lambda^{n-k}(\R^{n+1}).
  \end{equation*}
  \begin{proof}
    By definition, $\P_s^-\Lambda^k(\R^{n+1})=i_U\P_{s-1}\Lambda^{k+1}(\R^{n+1})$. Thus, we aim to show that any form that can be expressed as $*_{S^n}\hat\gamma$ for $\hat\gamma\in\P_{s-1}\Lambda^{n-k}(\R^{n+1})$ can be expressed as $i_U\hat\beta$ for $\hat\beta\in\P_{s-1}\Lambda^{k+1}(\R^{n+1})$, and vice versa.

    Because $\nu$ is dual to $U$ with respect to the standard metric, we use Proposition \ref{prop:interiorhodge} to compute that for $\hat\gamma\in\Lambda^{n-k}(\R^{n+1})$,
    \begin{equation*}
      *_{S^n}\hat\gamma=*_{\R^{n+1}}(\nu\wedge\hat\gamma)=(-1)^{n-k}*_{\R^{n+1}}(\hat\gamma\wedge\nu)=(-1)^{n-k}i_U(*_{\R^{n+1}}\hat\gamma).
    \end{equation*}
    Thus, setting $\hat\beta:=(-1)^{n-k}{*_{\R^{n+1}}\hat\gamma}$, we have $*_{S^n}\hat\gamma=i_U\hat\beta$. If $\hat\gamma\in\P_{s-1}\Lambda^{n-k}(\R^{n+1})$, then $\hat\beta\in\P_{s-1}\Lambda^{k+1}(\R^{n+1})$. Conversely, if $\hat\beta\in\P_{s-1}\Lambda^{k+1}(\R^{n+1})$, then $\hat\gamma=(-1)^{n-k}{*_{\R^{n+1}}^{-1}\hat\beta}\in\P_{s-1}\Lambda^{n-k}(\R^{n+1})$.
  \end{proof}
\end{proposition}

We immediately obtain the corresponding result for forms on $S^n$.
\begin{proposition}\label{prop:pmhodgesphere}
  The space $\P_s^-\Lambda^k(S^n)$ is the image under $*_{S^n}$ of differential forms of lower polynomial degree. That is,
  \begin{equation}
    \P_s^-\Lambda^k(S^n)=*_{S^n}\P_{s-1}\Lambda^{n-k}(S^n).\label{eq:pmhodgesphere}
  \end{equation}

  \begin{proof}
    The restriction of $\P_s^-\Lambda^k(\R^{n+1})$ to the sphere is $\P_s^-\Lambda^k(S^n)$ by definition. The restriction of $*_{S^n}\P_{s-1}\Lambda^{n-k}(\R^{n+1})$ to the sphere is $*_{S^n}\P_{s-1}\Lambda^{n-k}(S^n)$ because if $\hat\gamma$ is an extension of $\gamma$, then $*_{S^n}\hat\gamma$ is an extension of $*_{S^n}\gamma$.
  \end{proof}
\end{proposition}

\begin{corollary}\label{cor:pmhodgesphereevenodd}
  We have that
  \begin{align*}
    \P_s^-\Lambda^k_e(S^n)&=*_{S^n}\P_{s-1}\Lambda^{n-k}_o(S^n),\\
    \P_s^-\Lambda^k_o(S^n)&=*_{S^n}\P_{s-1}\Lambda^{n-k}_e(S^n).
  \end{align*}

  \begin{proof}
    We use Propositions \ref{prop:restrictevenodd} and \ref{prop:starevenodd} to take the even and odd parts of both sides of equation \eqref{eq:pmhodgesphere}.
  \end{proof}
\end{corollary}

We will also need the following characterization of $\P_s^-\Lambda^k(\R^{n+1})$.
\begin{proposition}\label{prop:pmker}
  Consider
  \begin{equation*}
    i_U\colon\P_s\Lambda^k(\R^{n+1})\to\P_{s+1}\Lambda^{k-1}(\R^{n+1}).
  \end{equation*}
  Then for $k\ge1$, 
  \begin{equation}\label{eq:keriu}
    \ker i_U=\P_s^-\Lambda^k(\R^{n+1})
  \end{equation}
  For $k=0$, assuming $s\ge0$,
  \begin{equation*}
    \ker i_U=\P_s\Lambda^0(\R^{n+1})=\P_s^-\Lambda^0(\R^{n+1})+\P_0\Lambda^0(\R^{n+1}).
  \end{equation*}
  If $s<0$, equation \eqref{eq:keriu} holds vacuously.
  \begin{proof}
    By definition $\P_s^-\Lambda^k(\R^{n+1})=i_U\P_{s-1}\Lambda^{k+1}(\R^{n+1})$. Since $i_U\circ i_U=0$, we conclude that $\P_s^-\Lambda^k(\R^{n+1})\subseteq\ker i_U$. Since $i_U$ vanishes on $0$-forms, we also have that $\P_s^-\Lambda^0(\R^{n+1})+\P_0\Lambda^0(\R^{n+1})\subseteq\P_s\Lambda^0(\R^{n+1})=\ker i_U$.

    Conversely, let $\hat\alpha\in\ker i_U$. Assuming $s\ge0$, decompose $\hat\alpha$ into homogeneous components. That is, write $\hat\alpha=\hat\alpha_s+\dotsb+\hat\alpha_0$, where the coefficients of $\hat\alpha_j$ are homogeneous polynomials of degree $j$. Observe that the coefficients of $i_U\hat\alpha_j$ are homogeneous polynomials of degree $j+1$. Thus, $i_U\hat\alpha=0$ implies that $i_U\hat\alpha_j=0$ for all $j$.

    We show that $\hat\alpha_j\in\P_s^-\Lambda^k(\R^{n+1})$ by showing that the Lie derivative $\mathfrak L_U\hat\alpha_j$ is in $\P_s^-\Lambda^k(\R^{n+1})$ and that $\mathfrak L_U\hat\alpha_j$ is a constant multiple of $\hat\alpha_j$.

    Using the Cartan formula, we have that
    \begin{equation*}
      \mathfrak L_U\hat\alpha_j=i_Ud\hat\alpha_j+di_U\hat\alpha_j=i_Ud\hat\alpha_j.
    \end{equation*}
    We have that $\hat\alpha_j\in\P_j\Lambda^k(\R^{n+1})\subseteq\P_s\Lambda^k(\R^{n+1})$, so $d\hat\alpha_j\in\P_{s-1}\Lambda^{k+1}(\R^{n+1})$, and so $\mathfrak L_U\hat\alpha_j\in\P_s^-\Lambda^k(\R^{n+1})$ by definition.

    To show that $\mathfrak L_U\hat\alpha_j$ is a constant multiple of $\hat\alpha_j$, we start with the fact that $\mathfrak L_Uu_i=u_i$. Because the Lie derivative commutes with the exterior derivative $d$, we then have $\mathfrak L_U(du_i)=du_i$. The form $\hat\alpha_j$ is a linear combination of products of $j$ terms of the form $u_i$ and $k$ terms of the form $du_i$. Thus, the product rule for the Lie derivative gives us that $\mathfrak L_U\hat\alpha_j=(j+k)\hat\alpha_j$. 

    Except for the case $j=k=0$, we can divide $\mathfrak L_U\hat\alpha_j$ by $j+k$ to conclude that $\hat\alpha_j\in\P_s^-\Lambda^k(\R^{n+1})$. When $k\ge1$, this means that $\hat\alpha_j\in\P_s^-\Lambda^k(\R^{n+1})$ for all $j$, so $\hat\alpha\in\P_s^-\Lambda^k(\R^{n+1})$. When $k=0$, we have that $\hat\alpha_j\in\P_s^-\Lambda^0(\R^{n+1})$ for $j\ge1$ and that $\hat\alpha_0\in\P_0\Lambda^0(\R^{n+1})$, so $\hat\alpha\in\P_s^-\Lambda^0(\R^{n+1})+\P_0\Lambda^0(\R^{n+1})$.
  \end{proof}
\end{proposition}

\section{Characterizations of \texorpdfstring{$\ooP_s\Lambda^k(S^n)$}{ooPs\textLambda k(Sn)} and \texorpdfstring{$\ooP_s^-\Lambda^k(S^n)$}{ooPs-\textLambda k(Sn)}}\label{sec:oops}
In Section \ref{subsec:oop}, we showed that $\Phi^*$ is an isomorphism between $\oP_r\Lambda^k(T^n)$ and $\ooP_s\Lambda^k_e(S^n)$, where $s=2r+k$. To complete the proof of the correspondence in Theorem \ref{thm:TSiso} for the $\oP_r\Lambda^k(T^n)$ and $\oP_r^-\Lambda^k(T^n)$ spaces, we show in this section that $\ooP_s\Lambda^k_e(S^n)=u_N\P_{s-n-1}\Lambda^k_o(S^n)$ and $\ooP_s^-\Lambda^k_e(S^n)=u_N\P_{s-n-1}^-\Lambda^k_o(S^n)$.

This section is structured as follows. We first prove the corresponding results for forms on $\R^{n+1}$, showing that forms in $\ooP_s\Lambda^k(\R^{n+1})$ and $\ooP_s^-\Lambda^k(\R^{n+1})$ are divisible by $u_N$ and determining the quotient. However, unlike in the previous section, the desired results for forms on $S^n$ do not immediately follow from these results for forms on $\R^{n+1}$. The missing ingredient is showing that forms in $\ooP_s\Lambda^k(S^n)$ can be extended to forms in $\ooP_s\Lambda^k(\R^{n+1})$ and that forms in $\ooP_s^-\Lambda^k(S^n)$ can be extended to forms in $\ooP_s^-\Lambda^k(\R^{n+1})$. This claim is subtle, and in fact fails for exceptional values of $s$ and $k$.

\begin{example}\label{eg:vdv}
Using coordinates $(u,v)$ on $\R^2$, consider $v\,dv$. This form is in $\ooP_1\Lambda^1(S^1)$ as seen in Example \ref{eg:udu}, but it cannot be extended to a form in $\ooP_1\Lambda^1(\R^2)$. Indeed, as we will show in Proposition \ref{prop:oopr}, a form in $\ooP_s\Lambda^1(\R^2)$ must be divisible by $uv$ and thus have degree at least two.
\end{example}

The situation is simpler for the $\P_s^-\Lambda^k(S^n)$ spaces compared to the $\P_s\Lambda^k(S^n)$ spaces: we will show that forms in $\ooP_s^-\Lambda^k(S^n)$ can always be extended to forms in $\ooP_s^-\Lambda^k(\R^{n+1})$. We can then use the result that forms in $\ooP_s^-\Lambda^k(\R^{n+1})$ are divisible by $u_N$ to get our desired result that forms in $\ooP_s^-\Lambda^k(S^n)$ are divisible by $u_N$. By taking the Hodge dual of this result, we can show that forms in $\ooP_s\Lambda^k(S^n)$ are also divisible by $u_N$, without needing to show that these forms can be extended to forms in $\ooP_s\Lambda^k(\R^{n+1})$.

\subsection{Differential forms on $\R^{n+1}$}
\begin{proposition}\label{prop:oopr}
  We have that
  \begin{equation*}
    \ooP_s\Lambda^k(\R^{n+1})=u_N\P_{s-n-1}\Lambda^k(\R^{n+1}).
  \end{equation*}
  \begin{proof}
    If $\hat\alpha\in\ooP_s\Lambda^k(\R^{n+1})$, then each polynomial coefficient of $\hat\alpha$ vanishes when we set $u_i=0$, so each coefficient is divisible by $u_i$. Thus, $\hat\alpha$ is divisible by $u_N$. Conversely, if $\hat\alpha\in u_N\P_{s-n-1}\Lambda^k(\R^{n+1})$, then $\hat\alpha\in\ooP_s\Lambda^k(\R^{n+1})$ because $u_N=0$ whenever $u_i=0$.
  \end{proof}
\end{proposition}

We prove the corresponding claim for $\ooP_s^-\Lambda^k(\R^{n+1})$, but there is an exception due to the fact that $u_N\in\ooP_{n+1}^-\Lambda^0(\R^{n+1})$ but $1\notin\P_0^-\Lambda^0(\R^{n+1})$.
\begin{proposition}\label{prop:oopmr}
  For $k\ge1$, we have that
  \begin{equation}\label{eq:oopmr}
    \ooP_s^-\Lambda^k(\R^{n+1})=u_N\P_{s-n-1}^-\Lambda^k(\R^{n+1}).
  \end{equation}
  When $k=0$, assuming $s\ge n+1$, we have that
  \begin{equation*}
    \ooP_s^-\Lambda^0(\R^{n+1})=u_N\left(\P_{s-n-1}^-\Lambda^0(\R^{n+1})+\P_0\Lambda^0(\R^{n+1})\right).
  \end{equation*}
  If $s<n+1$, equation \eqref{eq:oopmr} holds vacuously.
  \begin{proof}
    Consider
    \begin{equation*}
      i_U\colon\P_{s-n-1}\Lambda^k(\R^{n+1})\to\P_{s-n}\Lambda^{k-1}(\R^{n+1}).
    \end{equation*}
    By Proposition \ref{prop:pmker}, proving the above claim is equivalent to showing that
    \begin{equation*}
      \ooP_s^-\Lambda^k(\R^{n+1})=u_N\ker i_U.
    \end{equation*}

    Let $\hat\alpha\in\ooP_s^-\Lambda^k(\R^{n+1})$. Because $\hat\alpha\in\P_s^-\Lambda^k(\R^{n+1})$, we know that $i_U\hat\alpha=0$. Because $\hat\alpha\in\ooP_s\Lambda^k(\R^{n+1})$, we know by Proposition \ref{prop:oopr} that $\hat\alpha=u_N\hat\beta$ for $\hat\beta\in\P_{s-n-1}\Lambda^k(\R^{n+1})$. It remains to show that $\hat\beta\in\ker i_U$. We have
    \begin{equation}\label{eq:iualpha}
      i_U\hat\alpha=i_U(u_N\hat\beta)=u_N(i_U\hat\beta).
    \end{equation}
    Thus $u_N(i_U\hat\beta)=0$. Dividing by $u_N$, we conclude that $i_U\hat\beta=0$, as desired.

    Conversely, let $\hat\alpha=u_N\hat\beta$ where $\hat\beta\in\ker i_U$. Proposition \ref{prop:oopr} tells us that $\hat\alpha\in\ooP_s\Lambda^k(\R^{n+1})$. Meanwhile, equation \eqref{eq:iualpha} gives us that $i_U\hat\alpha=0$, so Proposition \ref{prop:pmker} tells us that $\hat\alpha\in\P_s^-\Lambda^k(\R^{n+1})$. (We can rule out the $\P_0\Lambda^0(\R^{n+1})$ summand in Proposition \ref{prop:pmker} because $\hat\alpha$ is divisible by $u_N$.)
  \end{proof}
\end{proposition}

\subsection{Extending forms on $S^n$ to forms on $\R^{n+1}$}
We would like to use these results for $\R^{n+1}$ to show the corresponding results for the sphere. It is clear that the restriction of a form in $\ooP_s\Lambda^k(\R^{n+1})$ to $S^n$ must be in $\ooP_s\Lambda^k(S^n)$. It is less clear that forms in $\ooP_s\Lambda^k(S^n)$ can be extended to forms in $\ooP_s\Lambda^k(\R^{n+1})$; in fact, in exceptional cases, they cannot, as in Example \ref{eg:vdv} and more generally in Example \ref{eg:noextend}. However, as we will prove in this subsection, for the $\ooP_s^-\Lambda^k(S^n)$ spaces, this extension result always holds; there are no exceptional cases.

We begin with understanding the situation when $k=0$, with an additional parity hypothesis that we later remove without difficulty.

\begin{lemma}\label{lemma:extendp}
  Let $p\in\ooP_s\Lambda^0(S^n)$. By definition, $p$ has an extension $\hat p'\in\P_s\Lambda^0(\R^{n+1})$. Assume that the total degree of every term of $\hat p'$ has the same parity, either all even or all odd. Then $p$ has an extension $\hat p\in\ooP_s\Lambda^0(\R^{n+1})$.
  \begin{proof}
    The main idea of the proof is that we obtain $\hat p$ by homogenizing $\hat p'$. That is, we multiply each term of $\hat p'$ by an appropriate power of $r^2:=u_1^2+\dotsb+u_{n+1}^2$.

    More precisely, we construct $\hat p$ as follows. Without loss of generality, assume that $s=\deg\hat p'$. We decompose $\hat p'$ into homogeneous components, so we have either of the two cases:
    \begin{align*}
      \hat p'&=\sum_{j=0}^{s/2}\hat p'_{2j},\text{ or}&\hat p'&=\sum_{j=0}^{(s-1)/2}\hat p'_{2j+1},
    \end{align*}
    where $\hat p'_j$ is a homogeneous polynomial of degree $j$. We multiply each term of $\hat p'$ by an appropriate power of $r^2$ in order to make the polynomial homogeneous. That is, set
    \begin{align*}
      \hat p&=\sum_{j=0}^{s/2}\hat p'_{2j}\left(r^2\right)^{s/2-j},\text{ or}&\hat p&=\sum_{j=0}^{(s-1)/2}\hat p'_{2j+1}\left(r^2\right)^{(s-1)/2-j}.
    \end{align*}

    Since $r^2=1$ on $S^n$, we see that the restriction of $\hat p$ to $S^n$ is the same as the restriction of $\hat p'$, namely $p$. We also see that $\hat p$ is a homogeneous polynomial of degree $s$. Consequently, if $\hat p(u)=0$, then $\hat p(\lambda u)=0$ for any real number $\lambda$. For any $u\in\Gamma_i\cap S^n$, we have $\hat p(u)=p(u)=0$ because $p\in\ooP_s\Lambda^0(S^n)$. Since every point in $\Gamma_i$ is of the form $\lambda u$ for $u\in\Gamma_i\cap S^n$, we conclude that $\hat p\in\ooP_s\Lambda^0(\R^{n+1})$, as desired.
  \end{proof}
\end{lemma}

We now prove this extension result for $\ooP_s^-\Lambda^k(S^n)$, again with an additional parity hypothesis that we remove later.
\begin{lemma}\label{lemma:extendalpha}
    Let $\alpha\in\ooP_s^-\Lambda^k(S^n)$. By definition, $\alpha$ has an extension $\hat\alpha'\in\P_s^-\Lambda^k(\R^{n+1})$. Assume that the total degree of every term of every polynomial coefficient of $\hat\alpha'$ has the same parity, either all even or all odd. Then $\alpha$ has an extension $\hat\alpha\in\ooP_s^-\Lambda^k(\R^{n+1})$.
  \begin{proof}
    Let $u\in\Gamma_i\cap S^n$, and consider the antisymmetric tensor $\hat\alpha'_u$. We show that $\hat\alpha'_u=0$. Indeed, for any vectors $V_1,\dotsc,V_k$ tangent to the sphere, we have $\hat\alpha'_u(V_1,\dotsc,V_k)=\alpha_u(V_1,\dotsc,V_k)=0$ because $\alpha\in\ooP_s\Lambda^k(S^n)$. Meanwhile, $U$ is the normal vector to the sphere, and $i_U\hat\alpha'_u=0$ because $\hat\alpha'\in\P_s^-\Lambda^k(\R^{n+1})$. We conclude by multilinearity that $\hat\alpha'_u=0$.

    Thus, all of the polynomial coefficients of $\hat\alpha'$ vanish on $\Gamma_i\cap S^n$ for all $i$.
    By Lemma \ref{lemma:extendp}, we can homogenize these polynomial coefficients to obtain polynomials that vanish on all of $\Gamma_i$ and have the same values on $S^n$, giving us an extension $\hat\alpha\in\ooP_s\Lambda^k(\R^{n+1})$.

    To show that $\hat\alpha\in\P_s^-\Lambda^k(\R^{n+1})$, for $u\in S^n$, we have that $i_U\hat\alpha_u=i_U\hat\alpha'_u=0$. Since $\hat\alpha$ is homogeneous, we conclude that $i_U\hat\alpha_u=0$ for all $u\in\R^{n+1}$. We then conclude that $\hat\alpha\in\ooP_s^-\Lambda^k(\R^{n+1})$ by Proposition \ref{prop:pmker}. (We can rule out the $\P_0\Lambda^0(\R^{n+1})$ summand of Proposition \ref{prop:pmker} using the fact that $\hat\alpha$ is homogeneous.)
  \end{proof}
\end{lemma}

We now remove the parity hypothesis.

\begin{proposition}\label{prop:extendalpha}
  Any differential form in $\ooP_s^-\Lambda^k(S^n)$ can be extended to a differential form in $\ooP_s^-\Lambda^k(\R^{n+1})$.
  \begin{proof}
    Let $\alpha\in\ooP_s^-\Lambda^k(S^n)$. Let
    \begin{align*}
      \beta_u&=\tfrac12\left(\alpha_u+\alpha_{-u}\right),\\\gamma_u&=\tfrac12\left(\alpha_u-\alpha_{-u}\right).
    \end{align*}
    Then $\beta$ and $\gamma$ satisfy the hypotheses of Lemma \ref{lemma:extendalpha}, because if $\hat\alpha'\in\P_s^-\Lambda^k(\R^{n+1})$ is an extension of $\alpha$, then $\hat\beta'_u=\frac12\left(\hat\alpha'_u+\hat\alpha'_{-u}\right)$ and $\hat\gamma'_u=\frac12\left(\hat\alpha'_u-\hat\alpha'_{-u}\right)$ are extensions of $\beta$ and $\gamma$ with the required parity property. Thus, we can extend $\beta$ and $\gamma$ to $\hat\beta$ and $\hat\gamma$ in $\ooP_s^-\Lambda^k(\R^{n+1})$ and set $\hat\alpha=\hat\beta+\hat\gamma$.
  \end{proof}
\end{proposition}

\subsection{Differential forms on $S^n$}
Now that we have this extension result, we can use the fact that forms in $\ooP_s^-\Lambda^k(\R^{n+1})$ are divisible by $u_N$ to conclude that the same holds for forms in $\ooP_s^-\Lambda^k(S^n)$.

\begin{proposition}\label{prop:oopm}
  If $k\ge1$, then
  \begin{equation}\label{eq:oopm}
    \ooP_s^-\Lambda^k(S^n)=u_N\P_{s-n-1}^-\Lambda^k(S^n).
  \end{equation}
  For $k=0$, assuming $s\ge n+1$, we have
  \begin{equation*}
    \ooP_s^-\Lambda^k(S^n)=u_N\left(\P_{s-n-1}^-\Lambda^0(S^n)+\P_0\Lambda^0(S^n)\right)
  \end{equation*}
  If $s<n+1$, then equation \eqref{eq:oopm} holds vacuously.

  Moreover, $\P_0\Lambda^0(S^n)\subseteq\P_2^-\Lambda^0(S^n)$, so equation \eqref{eq:oopm} holds even if $k=0$ as long as $s\ge n+3$.
  \begin{proof}
    Proposition \ref{prop:oopmr} gives the corresponding result for forms on $\R^{n+1}$. Restricting both sides to $S^n$ using Proposition \ref{prop:extendalpha}, we obtain the desired result.

    To show that $\P_0\Lambda^0(S^n)\subseteq\P_2^-\Lambda^0(S^n)$, observe that $i_U\nu=u_1^2+\dotsb+u_{n+1}^2$ is in $\P_2^-\Lambda^0(\R^{n+1})$ by definition and restricts to the constant function $1$ on $S^n$.
  \end{proof}
\end{proposition}

We take the even part of the above result.

\begin{corollary}\label{cor:oopmeven}
  For all $k$ and $s$, we have
  \begin{equation*}
    \ooP_s^-\Lambda^k_e(S^n)=u_N\P_{s-n-1}^-\Lambda^k_o(S^n).
  \end{equation*}
  \begin{proof}
    We use Propositions \ref{prop:restrictevenodd} and \ref{prop:unevenodd} to take the even part of the equations in Proposition \ref{prop:oopm}.
    The result holds even when $k=0$ because constant functions are even, so the odd part of $\P_0\Lambda^0(S^n)$ is zero.
  \end{proof}
\end{corollary}

We now use Hodge duality to prove the corresponding result for $\ooP_s\Lambda^k(S^n)$.

\begin{lemma}\label{lemma:oopmhodgesphere}
  We have
  \begin{equation*}
    \ooP_s^-\Lambda^k(S^n)=*_{S^n}\ooP_{s-1}\Lambda^{n-k}(S^n).
  \end{equation*}
  \begin{proof}
    Proposition \ref{prop:pmhodgesphere} gives us $\P_s^-\Lambda^k(S^n)=*_{S^n}\P_{s-1}\Lambda^{n-k}(S^n)$, and we have $\alpha_u=0$ if and only if $*_{S^n}\alpha_u=0$.
  \end{proof}
\end{lemma}

\begin{proposition}\label{prop:extendoop}
  If $k<n$, then
  \begin{equation}\label{eq:extendoop}
    \ooP_s\Lambda^k(S^n)=u_N\P_{s-n-1}\Lambda^k(S^n),
  \end{equation}

  If $k=n$, assuming $s\ge n$, we have
  \begin{equation*}
    \ooP_s\Lambda^n(S^n)=u_N\left(\P_{s-n-1}\Lambda^n(S^n)+\R\cdot\vol_{S^n}\right)
  \end{equation*}
  If $s<n$, then equation \eqref{eq:extendoop} holds vacuously.
  
  Moreover, $\vol_{S^n}\in\P_1\Lambda^n(\R^{n+1})$, so equation \eqref{eq:extendoop} holds even if $k=n$ as long as $s\ge n+2$.

  \begin{proof}
    When $k<n$ or $s<n$, Proposition \ref{prop:oopm} gives us
    \begin{equation*}
      \ooP_{s+1}^-\Lambda^{n-k}(S^n)=u_N\P_{s-n}^-\Lambda^{n-k}(S^n).
    \end{equation*}
    Lemma \ref{lemma:oopmhodgesphere} and Proposition \ref{prop:pmhodgesphere} characterize the above spaces as images under $*_{S^n}$, giving us
    \begin{equation*}
      *_{S^n}\ooP_s\Lambda^k(S^n)=u_N\left(*_{S^n}\P_{s-n-1}\Lambda^k(S^n)\right)
    \end{equation*}
    Applying $*_{S^n}$ to both sides and using the fact that $*_{S^n}(*_{S^n}\alpha)=(-1)^{k(n-k)}\alpha$, we obtain
    \begin{equation*}
      \ooP_s\Lambda^k(S^n)=u_N\P_{s-n-1}\Lambda^k(S^n).
    \end{equation*}

    Meanwhile, in the case $k=n$ and $s\ge n$, Proposition \ref{prop:oopm} instead gives us
    \begin{equation*}
      \begin{split}
        \ooP_{s+1}^-\Lambda^0(S^n)&=u_N\left(\P_{s-n}^-\Lambda^0(S^n)+\P_0\Lambda^0(S^n)\right)\\
        &=u_N\left(\P_{s-n}^-\Lambda^0(S^n)+\mathbb R\cdot1\right)
      \end{split}
    \end{equation*}
    Taking the Hodge star as above, we instead obtain
    \begin{equation*}
      \ooP_s\Lambda^n(S^n)=u_N\left(\P_{s-n-1}\Lambda^n(S^n)+\mathbb R\cdot\vol_{S^n}\right).
    \end{equation*}

    To see that $\vol_{S^n}\in\P_1\Lambda^n(S^n)$, we have that $\vol_{S^n}=*_{S^n}1$, which means that $\vol_{S^n}$ is the restriction to $S^n$ of $*_{\R^{n+1}}\nu\in\P_1\Lambda^n(\R^{n+1})$.
  \end{proof}
\end{proposition}

\begin{corollary}\label{cor:oopeven}
  Assume that $k<n$ or $s\ge n+2$. Then
  \begin{equation*}
    \ooP_s\Lambda^k_e(S^n)=u_N\P_{s-n-1}\Lambda^k_o(S^n).
  \end{equation*}
  \begin{proof}
    We use Propositions \ref{prop:restrictevenodd} and \ref{prop:unevenodd} to take the even part of equation \eqref{eq:extendoop}.
  \end{proof}
\end{corollary}

\subsection{Additional remarks}
The above results suffice for proving our theorems, but for completeness we present the extension result for $\ooP_s\Lambda^k(S^n)$, analogous to the result we proved for $\ooP_s^-\Lambda^k(S^n)$ in Proposition \ref{prop:extendalpha}.

\begin{corollary}\label{cor:oop}
  Assume that $k<n$ or $s\ge n+2$. Then any form in $\ooP_s\Lambda^k(S^n)$ can be extended to a form in $\ooP_s\Lambda^k(\R^{n+1})$.
  \begin{proof}
    Proposition \ref{prop:extendoop} gives us that $\ooP_s\Lambda^k(S^n)=u_N\P_{s-n-1}\Lambda^k(S^n)$. By definition, any form in $u_N\P_{s-n-1}\Lambda^k(S^n)$ can be extended to a form in $u_N\P_{s-n-1}\Lambda^k(\R^{n+1})$, and this space is equal to $\ooP_s\Lambda^k(\R^{n+1})$ by Proposition \ref{prop:oopr}.
  \end{proof}
\end{corollary}

The conditions on $k$ and $s$ in Corollary \ref{cor:oop} are necessary, as we see from the following example.

\begin{example}\label{eg:noextend}
   Let $\alpha=u_N\vol_{S^n}$. As discussed above, $\vol_{S^n}$ can be extended to $*_{\R^{n+1}}\nu\in\P_1\Lambda^n(\R^{n+1})$, so $\alpha$ has an extension that is in $\ooP_{n+2}\Lambda^n(\R^{n+1})$. One can show that the degree $n+2$ is optimal due to the fact that any form in $\ooP_s\Lambda^n(\R^{n+1})$ must be divisible by $u_N$ and the fact that $\alpha$ is even. However, $\alpha$ itself is actually of lower degree. Indeed, $\alpha$ can be extended to $\hat\alpha=u_2\dotsm u_{n+1}\,du_2\wedge\dotsb\wedge du_{n+1}$, putting $\alpha$ in $\ooP_n\Lambda^n(S^n)$. We can check that $\hat\alpha$ does in fact restrict to $\alpha$ by computing that $\nu\wedge\hat\alpha=u_N\vol_{\R^{n+1}}$, and hence $*_{S^n}\hat\alpha=u_N$. See also Examples \ref{eg:vdv} and \ref{eg:scalarfields0}.
\end{example}

This example also explains the seemingly paradoxical claim that $\ooP_n\Lambda^n(S^n)=u_N(\R\cdot\vol_{S^n})$ given by Proposition \ref{prop:extendoop} when $k=n$ and $s=n$: we have a differential form of degree $n$ that is divisible by $u_N$, even though $u_N$ has degree $n+1$. See also Example \ref{eg:prsn}.

\section{Conclusion} The transformation $x_i=u_i^2$ induces a correspondence between differential forms on the simplex $x_1+\dotsb x_{n+1}=1$ and differential forms the sphere $u_1^2+\dotsb+u_{n+1}^2=1$. We completely characterized the spaces of differential forms on the sphere corresponding to the $\P_r\Lambda^k(T^n)$, $\P_r^-\Lambda^k(T^n)$, $\oP_r\Lambda^k(T^n)$, and $\oP_r^-\Lambda^k(T^n)$ families. This correspondence gives an explanation for the isomorphisms $\P_r\Lambda^k(T^n)\cong\oP_{r+k+1}^-\Lambda^{n-k}(T^n)$ and $\P_r^-\Lambda^k(T^n)\cong\oP_{r+k}\Lambda^{n-k}(T^n)$ used by Arnold, Falk, and Winther in their development of finite element exterior calculus: after the change of coordinates, both of these isomorphisms reveal themselves to be the Hodge star $*_{S^n}$ followed by multiplication by the bubble function $u_1\dotsm u_{n+1}$. Our result thus gives new isomorphism maps $\P_r\Lambda^k(T^n)\to\oP_{r+k+1}^-\Lambda^{n-k}(T^n)$ and $\P_r^-\Lambda^k(T^n)\to\oP_{r+k}\Lambda^{n-k}(T^n)$ that are defined pointwise. As seen in our examples, evaluating these maps is a quick computation that does not require expressing the differential forms in terms of a basis for $\P_r\Lambda^k(T^n)$ or $\P_r^-\Lambda^k(T^n)$.

\subsection{Future work}
One advantage of working with even differential forms on $S^n$ is that $S^n$ is a smooth manifold without boundary, in contrast to $T^n$, which has a boundary that is not smooth. Potentially, this feature could be used to simplify the functional analysis arguments used in the development of finite element exterior calculus, such as Poincar\'e's inequality discussed in \cite{afw06} and the bounded cochain projections discussed in \cite{afw06} and \cite{fw14}.

Another direction in which one could continue this work would be to consider parallelotope meshes instead of simplicial meshes. Finite element exterior calculus for parallelotope meshes \cite{aa11,abb15} works in much the same way as for simplicial meshes. Following the ideas in our current work, one could apply a coordinate transformation to convert a differential form on an $n$-dimensional parallelotope into an even differential form on an $n$-torus. Could such a coordinate transformation result in a better understanding of finite element spaces such as the serendipity elements?

\section{Acknowledgments}
I would like to thank Ari Stern, Douglas Arnold, Evan Gawlik, and the anonymous referees for their feedback on this project. I would also like to acknowledge the support of the AMS--Simons Travel Grant.

\bibliographystyle{siam}
\bibliography{fem}

\appendix
\section{Equivalence of definitions of \texorpdfstring{$\mathcal P_r^-\Lambda^k(T^n)$}{Pr-\textLambda k(Tn)}}\label{sec:prmdef}
The definition of $\mathcal P_r^-\Lambda^k(T^n)$ that we gave in Definition \ref{def:pm} is different from the definition given by Arnold, Walk, and Winther in \cite{afw06}. In this appendix, we show that the two definitions are equivalent.

The definition in \cite[Subsection 3.2]{afw06} involves choosing an arbitrary point $z$ in the simplex $T^n$ and defining a vector field on $T^n$ such that the vector based at $x\in T^n$ points away from $z$ with magnitude $\abs{x-z}$. Setting $z$ to be the center of the simplex $\left(\frac1{n+1},\dotsc,\frac1{n+1}\right)$, one can check that this vector field is given by the formula
\begin{equation*}
  V_\kappa=\sum_{i=1}^{n+1}\left(x_i-\tfrac1{n+1}\right)\pp{x_i}
\end{equation*}
for $x\in T^n$. Using this vector field, we can define $\P_r^-\Lambda^k(T^n)$.

\begin{definition}\label{def:pmafw}
  For $r\ge1$, Arnold, Falk, and Winther \cite{afw06} define $\P_r^-\Lambda^k(T^n)$ to be
  \begin{equation*}
    \P_{r-1}\Lambda^k(T^n)+\kappa\P_{r-1}\Lambda^{k+1}(T^n),
  \end{equation*}
  where $\kappa$ denotes the interior product with the vector field $V_\kappa$.
\end{definition}

\begin{proposition}\label{prop:prmdfn}
  The definitions of $\P_r^-\Lambda^k(T^n)$ given by Definitions \ref{def:pm} and \ref{def:pmafw} are equivalent.

\begin{proof}
  Let $t=x_1+\dotsb+x_{n+1}$. With the equation
  \begin{equation*}
    V_\kappa:=\sum_{i=1}^{n+1}\left(x_i-\tfrac t{n+1}\right)\pp{x_i}=X-\tfrac t{n+1}\nabla t,
  \end{equation*}
  we can extend the definition of $V_\kappa$ to all of $\R^{n+1}$. Geometrically, $V_\kappa$ is the projection of the radial vector field $X$ to the simplicies $t=\text{const}$. We can extend the definition of $\kappa$ to $\kappa\colon\Lambda^{k+1}(\R^{n+1})\to\Lambda^k(\R^{n+1})$ to be the interior product with the vector field $V_\kappa$, and so we have
  \begin{equation*}
    \kappa=i_X-\tfrac t{n+1}i_{\nabla t}.
  \end{equation*}

  Assume now that $a$ satisfies Definition \ref{def:pm}, so $a$ can be extended to a differential form $\hat a=i_X\hat b$, where $\hat b\in\P_{r-1}\Lambda^{k+1}(\R^{n+1})$. Then
  \begin{equation*}
    \hat a=\kappa\hat b+\tfrac t{n+1}i_{\nabla t}\hat b
  \end{equation*}
  Thus $a=\kappa b+c$, where $b$ and $c$ are the restrictions of $\hat b$ and $\tfrac t{n+1}i_{\nabla t}\hat b$ to $T^n$, respectively. By definition, $b\in\P_{r-1}\Lambda^{k+1}(T^n)$. Meanwhile, since $t=1$ on $T^n$, we know that $c$ is also the restriction to $T^n$ of $\tfrac1{n+1}i_{\nabla t}\hat b$, which is in $\P_{r-1}\Lambda^k(\R^{n+1})$ because $\nabla t$ has polynomial degree zero. Thus $c\in\P_{r-1}\Lambda^k(T^n)$. We conclude that $a\in\kappa\P_{r-1}\Lambda^{k+1}(T^n)+\P_{r-1}\Lambda^k(T^n)$, so $a$ satisfies Definition \ref{def:pmafw}.

  Conversely, assume that $a$ satisfies Definition \ref{def:pmafw}. We consider the two summands as separate cases. 

  If $a\in\P_{r-1}\Lambda^k(T^n)$, then let $\hat a'\in\P_{r-1}\Lambda^k(\R^{n+1})$ be an arbitrary extension of $a$, and let
  \begin{align}
    \hat a&:=i_X(dt\wedge\hat a')\label{eq:projecta1}\\
          &\phantom:=i_X(dt)\hat a'-dt\wedge i_X\hat a'\\
          &\phantom:=t\hat a'-dt\wedge i_X\hat a'.\label{eq:projecta2}
  \end{align}
  Using equation \eqref{eq:projecta2}, since the restriction of $t$ to $T^n$ is $1$ and the restriction of $dt$ to $T^n$ is zero, the restriction of $\hat a$ to $T^n$ is the same as the restriction of $\hat a'$, namely $a$. Meanwhile, using equation \eqref{eq:projecta1}, we have that $\hat a\in i_X\P_{r-1}\Lambda^{k+1}(\R^{n+1})$ because $dt\in\P_0\Lambda^1(\R^{n+1})$, so $dt\wedge\hat a'\in\P_{r-1}\Lambda^{k+1}(\R^{n+1})$. Thus $a$ satisfies Definition \ref{def:pm}.

  If $a=\kappa b$ for some $b\in\P_{r-1}\Lambda^{k+1}(T^n)$, then let $\hat b'\in\P_{r-1}\Lambda^{k+1}(\R^{n+1})$ be an arbitrary extension of $b$. Set
  \begin{align}
    \hat b&:=\tfrac1{n+1}i_{\nabla t}(dt\wedge\hat b')\label{eq:projectb1}\\
          &\phantom:=\tfrac1{n+1}i_{\nabla t}(dt)\hat b'-\tfrac1{n+1}dt\wedge i_{\nabla t}\hat b'\\
          &\phantom:=\hat b'-\tfrac1{n+1}dt\wedge i_{\nabla t}\hat b'.\label{eq:projectb2}
  \end{align}
  As before, observe from equation \eqref{eq:projectb2} that the restriction of $\hat b$ to $T^n$ is the same as the restriction of $\hat b'$, namely $b$. Because $V_\kappa$ is tangent to $T^n$, we can then conclude that the restriction of $\kappa\hat b$ to $T^n$ is $\kappa b=a$. Next, observe from equation \eqref{eq:projectb1} that $i_{\nabla t}\hat b=0$. Consequently, we can set
  \begin{equation*}
    \hat a:=i_X\hat b=\kappa\hat b+\tfrac t{n+1}i_{\nabla t}\hat b=\kappa\hat b,
  \end{equation*}
  from which we see that $\hat a$ is an extension of $a$. Finally, observe that $\hat b\in\P_{r-1}\Lambda^{k+1}(\R^{n+1})$ because $dt$ and $\nabla t$ both have polynomial degree zero. Thus $a$ satisfies Definition \ref{def:pm}, as desired.  
\end{proof}
\end{proposition}

\section{Vector space identities}
In this appendix, we provide proofs of two vector space identities involving the Hodge star. I believe that these identities are well-known, but I have not been able to find published references for them.

Let $V$ be a vector space equipped with an inner product and an orientation. This structure defines an inner product $\langle\cdot,\cdot\rangle$ on the exterior algebra $\bigwedge^*V^*$ and a Hodge star map $*\colon\bigwedge^*V^*\to\bigwedge^*V^*$.

\subsection{The interior product, the exterior product, and the Hodge star} Let $X\in V$, and let $\nu\in V^*$ be the dual vector corresponding to $X$ with respect to the inner product. It is a standard result that the adjoint of $i_X$ is $\nu\wedge$ in the sense that
\begin{equation*}
  \langle i_X\hat\alpha,\hat\beta\rangle=\langle\hat\alpha,\nu\wedge\hat\beta\rangle,
\end{equation*}
where $\hat\alpha\in\bigwedge^kV^*$ and $\hat\beta\in\bigwedge^{k-1}V^*$.

This adjoint relationship has the following consequence for the Hodge star on $\bigwedge^*V^*$.

\begin{proposition}\label{prop:interiorhodge}
    Let $V$ be an oriented inner product space. If $X\in V$ and $\nu\in V^*$ are dual to one another with respect to the inner product, then
  \begin{equation*}
    i_X(*\hat\alpha)=*(\hat\alpha\wedge\nu).
  \end{equation*}
  for all $\hat\alpha\in\bigwedge^kV^*$. 
  \begin{proof}
    Let $\dim V=n+1$. For all $\hat\beta\in\bigwedge^{n-k}V^*$, we have
    \begin{multline*}
      \langle i_X(*\hat\alpha),\hat\beta\rangle\vol=\langle *\hat\alpha,\nu\wedge\hat\beta\rangle\vol=\langle\hat\alpha,*^{-1}(\nu\wedge\hat\beta)\rangle\vol\\
      =\hat\alpha\wedge\nu\wedge\hat\beta
      =\langle\hat\alpha\wedge\nu,*^{-1}\hat\beta\rangle\vol=\langle*(\hat\alpha\wedge\nu),\hat\beta\rangle\vol.\qedhere
    \end{multline*}
  \end{proof}
\end{proposition}

\subsection{The Hodge star on a hyperplane}
In this subsection, we consider a hyperplane $H$ that is orthogonal to a unit covector $\nu\in V^*$. The inner product on $V$ induces an inner product on $H$, and the orientation on $V$ along with the choice of unit conormal $\nu$ induces an orientation on $H$. Thus, in addition to a Hodge star operator $\bigwedge^*V^*\to\bigwedge^*V^*$, we also have a Hodge star operator $\bigwedge^*H^*\to\bigwedge^*H^*$. We denote these by $*_V$ and $*_H$, respectively.

We show that $*_H$ can be computed from $*_V$ and $\nu$ as follows.

\begin{proposition}\label{prop:hodgehyperplane}
  With notation as above, let $\alpha\in\bigwedge^kH^*$ be the restriction to $H$ of $\hat\alpha\in\bigwedge^kV^*$. Then $*_H\alpha$ is the restriction to $H$ of
  \begin{equation*}
    *_V(\nu\wedge\hat\alpha).
  \end{equation*}
  \begin{proof}
    Let $n=\dim H$, let $\hat\beta\in\bigwedge^{n-k}V^*$ denote $*_V(\nu\wedge\hat\alpha)$, and let $\beta\in\bigwedge^{n-k}H^*$ be the restriction of $\hat\beta$ to $H$. We prove that $*_H\alpha=\beta$ by verifying it on a basis for $\hat\alpha$. 

    Let $\nu,e_1,\dotsc,e_n$ be an oriented orthonormal basis for $V^*$, so $e_1,\dotsc,e_n$ is an oriented orthonormal basis for $H^*$, and we have that $\vol_H=e_1\wedge\dotsb e_n$ and $\vol_V=\nu\wedge\vol_H$. For $I=\{i_1<i_2<\dotsb<i_k\}\subseteq\{1,\dotsc,n\}$, let $e_I$ denote $e_{i_1}\wedge\dotsb\wedge e_{i_k}$.

    If $\hat\alpha=\nu\wedge e_I$, then $\alpha=0$ because the restriction of $\nu$ to $H$ is zero. Also, $\nu\wedge\hat\alpha=0$, so $\hat\beta$ and $\beta$ are zero as well, as desired.

    Now let $\hat\alpha=e_I$. By the definition of $*_V$ with respect to the oriented basis $\nu,e_1,\dotsc,e_n$, we see that $\hat\beta=*_V(\nu\wedge e_I)$ has the form $\pm e_J$, where $J=\{1,\dotsc,n\}\setminus I$ and the sign is chosen so that $(\nu\wedge e_I)\wedge(\pm e_J)=\vol_V$. Since $\nu\wedge(e_I\wedge\pm e_J)=\vol_V$, we conclude that $e_I\wedge\pm e_J=\vol_H$. Thus, by the definition of $*_H$ with respect to the oriented basis $e_1,\dotsc,e_n$, we have $*_He_I=\pm e_J$, that is, $*_H\alpha=\beta$.
  \end{proof}
\end{proposition}

\end{document}